\documentclass[3p,11pt,letterpaper, reqno]{amsart}
\usepackage{mathrsfs}
\usepackage{amsmath,amssymb}

\usepackage{color} 
\usepackage{epsfig}
\usepackage{graphicx}
\usepackage[english]{babel}
\usepackage{algorithm,algorithmic}
\usepackage{makrosam}

\numberwithin{equation}{section}

\begin{document}


  \title[MLEnKF for spatially extended models]{Multilevel ensemble Kalman filtering for spatially extended models} 

\author[A. Chernov]{Alexey Chernov$^{\dagger\dagger}$}
\thanks{$^{\dagger\dagger}$Institute for Mathematics, Carl von Ossietzky University Oldenburg, Germany (alexey.chernov@uni-oldenburg.de)}

\author[H. Hoel]{H{\aa}kon Hoel$^*$}
\thanks{$^*$Department of Mathematics, University of Oslo, Norway (haakonah@math.uio.no)}

\author[K. J. H. Law]{Kody J. H. Law$^\dagger$}
\thanks{$^\dagger$Computer Science and Mathematics Division, 
Oak Ridge National Laboratory (lawkj@ornl.gov)}

\author[F. Nobile]{Fabio Nobile$^{\circ\circ}$}
\thanks{$^{\circ\circ}$Mathematics Institute of Computational Science and Engineering, \'Ecole polytechnique f\'ed\'erale de Lausanne, 
Switzerland (fabio.nobile@epfl.ch)}

\author[R. Tempone]{Raul Tempone$^\circ$}
\thanks{$^\circ$Applied Mathematics and Computational Sciences, 
KAUST, Thuwal, Saudi Arabia (raul.tempone@kaust.edu.sa)}

%

%




\begin{abstract}

  This work embeds a multilevel Monte Carlo (MLMC) sampling strategy
  into the Monte Carlo step of the ensemble Kalman filter (EnKF),
  thereby yielding a multilevel ensemble Kalman filter (MLEnKF) which
  has provably superior asymptotic cost to a given accuracy level.
  {The development of MLEnKF for finite-dimensional state-spaces in
    the work \cite{ourmlenkf} is here extended to models with
  infinite-dimensional state-spaces in the form of spatial fields.}
  {A concrete example is given to illustrate the results.}
 
\end{abstract}

\maketitle




\section{Introduction}
\label{sec:intro}


Filtering refers to the sequential estimation of the state $v$ and/or parameters $p$ of a system through sequential 
incorporation of online data $y$.  The most complete estimation of the state $v_n$ at time $n$ is given by its 
probability distribution
conditional on the observations up to the given time $\bbP(dv_n|y_1,\ldots, y_n)$ \cite{jaz70, BC09}.   
For linear Gaussian systems the analytical solution may be given in closed form, 
via update formulae for the mean and covariance known as the Kalman filter \cite{kalman1960new}.
However, in general there is no closed form solution.  
One must therefore resort to either algorithms which approximate 
the probabilistic solution by leveraging ideas from control theory 
\cite{kal03,jaz70}, or Monte Carlo methods to approximate the filtering distribution itself 
\cite{BC09, doucet2000sequential, del2004feynman}.  The ensemble Kalman filter (EnKF) 
\cite{burgers1998analysis, evensen2003ensemble} combines elements of both approaches.
In the linear Gaussian case it converges to the Kalman filter solution \cite{mandel2011convergence}, and even 
in the nonlinear case, under suitable assumptions it converges \cite{le2011large, law2014deterministic} to 
a limit which, for a single update, is optimal among those which incorporate the data linearly 
\cite{law2014deterministic, luenberger1968optimization, pajonk2012deterministic}.
In the case of spatially extended models approximated on a numerical grid, 
the state space itself may become very high-dimensional and even the linear solves may become intractable.
Therefore, one may be inclined to use the EnKF filter even for linear Gaussian problems in which the solution is
intractable despite being given in closed form on paper by the Kalman filter.

Herein the underlying problem will admit a hierarchy of approximations with cost inversely proportional to accuracy, and it will be necessary to approximate the target for a single prediction step.
It has been proposed to use a multilevel identity to optimize the work required to achieve a certain 
total error level in the Monte Carlo approximation of such random fields 
\cite{heinrich2001multilevel}.
See \cite{Giles14}
for a recent review of multilevel Monte Carlo (MLMC). 
Very recently, 
a number of works have emerged which extend the 
MLMC framework 
to the context of Monte Carlo algorithms designed for Bayesian inference. 
Examples include
Markov chain Monte Carlo 
\cite{ketelsen2013hierarchical, hoang2013complexity}, 
sequential Monte Carlo samplers 
\cite{beskos2015multilevel, jasra2016forward, del2016multilevel}, 
particle filters \cite{jasra2015multilevel, gregory2016multilevel}, 
and EnKF \cite{ourmlenkf}.  
The filtering papers \cite{jasra2015multilevel, gregory2016multilevel, ourmlenkf} 
thusfar all consider only \hh{finite-dimensional}
SDE forward models, with the approximation error as arising from time discretization.

The present work considers the extension of the multilevel EnKF (MLEnKF) 
\cite{ourmlenkf}
to spatially extended models.  
The infinite-dimensional case was considered in the context of the square root EnKF in 
\cite{kwiatkowski2014convergence}.  
As in that work, we will require that the limiting covariance is trace-class. 
It was mentioned above that the limiting EnKF distribution, 
the so-called mean-field EnKF (MFEnKF), 
is in general not the Bayesian posterior filtering 
distribution and has a fixed bias.  
The error of the EnKF approximation may be decomposed into 
MC error and this Gaussian bias as shown in \cite{law2014deterministic}.
According to folklore, small sample sizes are suitable, 
and it may well be due to minimum error being limited by the bias.  
Nonetheless, the latter is difficult to quantify and deal with, 
while the MC error can be controlled and minimized.  
Unfortunately, scientists are often limited to small ensemble sizes anyway, 
due to an extremely high-dimensional underlying state space, 
which is approximating a spatial field.  Within the MLEnKF framework
developed here, a much smaller MC error can be obtained for the same fixed cost,
which will lower the cost requirement for practitioners to ensure that the 
MC error is commensurate with the bias.
Furthermore, it has been shown in 
\cite{kelly2014well, tong2016nonlinear, tong2015nonlinear,kelly2015concrete}
that signal tracking stability of EnKF is based on a feedback control mechanism,
which can be established for a single member ensemble in 3DVAR \cite{brett2013accuracy, blomker2013accuracy, stuart2013analysis, tarn1976observers, olson2003determining, hayden2011discrete, farhat2015data}.
The greater accuracy of EnKF in comparison to 3DVAR \cite{law2015data}
is afforded presumably by its use of the ensemble statistics, and the relation to the 
optimal linear update.  Therefore, it is of interest to improve the MC approximation.


The rest of the paper will be organized as follows.  
In section \ref{sec:kalman} the notation and problem will be introduced,
and the spatial multilevel EnKF (MLEnKF) will be introduced for the first time in sub-section \ref{subsec:mlenkf}.  
In section \ref{sec:theory} it is proven that indeed the spatial MLEnKF inherits almost the same favorable asymptotic ``cost-to-$\varepsilon$'' as 
the standard MLMC for a finite time horizon, and its mean-field limiting distribution is the filtering distribution in the linear and Gaussian case. 
In section \ref{sec:example} a concrete example will be given to illustrate the theory.
Finally, conclusions and future directions are presented in section
\ref{sec:conclusion}.




\section{Kalman filtering}
\label{sec:kalman}


\subsection{General set-up}
\label{subsec:genFiltering}

Let $(\Omega,\cE, \bbP)$ be a complete probability space, where 
$\Omega$ is the set of events, $\cE$ is the sigma algebra of subsets of $\Omega$
and $\bbP$ is the associated probability measure.
Let $\cH$ be a separable Hilbert space and 
$L^p(\Omega;\cH) = \{u : \Omega \rightarrow \cH ; \bbE \|u\|^p_\cH <\infty \}$,
and denote the associated norm 
$\| u \|_{L^p(\Omega;\cH)} = (\bbE \|u\|^p_\cH)^{1/p}$, or just 
$\| u \|_p$ where the meaning is clear.
Consider the general stochastic signal evolution for the random variables 
$u_n \in L^p(\Omega;\cH)$, 
where, 
\begin{equation}\label{eq:psiDefinition}
u_{n+1} = \Psi(u_n), 
\end{equation}
for $n=0,1,\ldots,N-1$.  
In particular, we will be concerned herein with the case in which 
$\Psi: L^p(\Omega;\cH) \rightarrow L^p(\Omega;\cH)$
is the finite-time evolution of an SPDE or, equivalently, a discrete
random mapping (possibly nonlinear) of a spatially extended state 
given as a random $L^p$ integrable element of the separable Hilbert space $\cH$. 
Let $\{\phi_k\}_{k=1}^\infty$ be a countable
orthonormal basis spanning the Hilbert space $\cH$, so that elements
$u\in \cH$ admit the representation $u = \sum_{k=1}^\infty u^k
\phi_k$, where $u^k=\langle u,\phi_k \rangle_{\cH}$.  The notation
$\langle \cdot,\cdot \rangle_{\cH}$ and $\cdot \otimes \cdot$ is used
to denote the inner and outer products over $\cH$, with the induced
norm $\| \cdot \|_{\cH} \coloneq \langle \cdot, \cdot
  \rangle_{\cH}^{1/2}$, while for \hh{finite-dimensional} spaces we assign the
notation $\cR_d = (\bbR^d, \langle \cdot, \cdot \rangle)$ to denote the
Hilbert space with the Euclidean inner product and the induced norm
$\| \cdot \|_{\cR_d} \coloneq \langle \cdot, \cdot
  \rangle^{1/2}$.  Where required, the spatial variable will be
denoted with $x,z \in \bbR^d$ for some $0<d<\infty$.  


Given the history of 
signal observations
\begin{equation}
y_n = H u_n + \eta_n, 
\label{eq:obdef}
\end{equation}
with $H: \cH \rightarrow \bbR^m$ 
linear and $\eta_n$ are i.i.d. with $\eta_1 \sim N(0,\Gamma), \Gamma \in \bbR^{m\times m}$ symmetric positive definite, 
the objective is to track the signal $u_n$ given the observations 
$Y_n$ where $Y_n = (y_1,y_2, \ldots, y_n)$.
Notice that under the given assumptions 
we have a hidden Markov model.  That is, the distribution of the random variable we seek to approximate 
admits the following sequential structure
\begin{align}
\label{eq:filteringdist}
\bbP(du_n | Y_n) & = 
\frac1{Z(Y_n)}{\cL(u_n;y_n) \bbP(du_n |Y_{n-1})}, \\
\nonumber
\bbP(du_n |Y_{n-1}) & = \int_{u_{n-1}\in \cH} \bbP(du_n|u_{n-1})\bbP(du_{n-1}|Y_{n-1}), \\
\nonumber
\cL(u_n;y_n) & = \exp\{-\frac12\|\Gamma^{-1/2}(y_n-Hu_n)\|_{\cR_d}^2\}, \\
Z(Y_n) & = \int_{u_{n}\in \cH} \cL(u_n;y_n) \bbP(du_n |Y_{n-1}).
\nonumber
\end{align}
It will be assumed that $\Psi(\cdot)$ cannot be evaluated exactly, but rather only approximately,
and that there exists a hierarchy of accuracies at which it can be evaluated, each with its associated cost.
The explicit dependence on $\omega$ will be suppressed where confusion is not possible.  
For notational simplicity, we will consider the particular case in which 
the 
map 
$\Psi(\cdot)$ does not depend on $n$.  Note that the results easily extend to the
non-autonomous case, provided the given assumptions
on $\Psi$ are uniform with respect to $\{\Psi_n\}_{n=1}^N$. 
The specialization is merely for notational convenience.
In particular, we will need to denote by $\{\psiL \}_{\ell=0}^\infty$ a hierarchy of approximations to the solution 
$\Psi:=\Psi^{\infty}$. 
First some assumptions must be made.
 \begin{assumption}  For every $p \geq 2$, 
 the solution operators $\{\psiL\}_{\ell=0}^\infty$ satisfy the following conditions, for
 some $0<c_\Psi<\infty$ depending on $\Psi$:
\begin{itemize} 
\item[(i)] $\|\psiL(u) -\psiL(v) \|_{L^p(\Omega;\cH)} {\le} c_{\Psi} \|u-v\|_{L^p(\Omega;\cH)} $, 
\item[(ii)] $\|\psiL(u)\|_{L^p(\Omega;\cH)}^p \leq c_{\Psi} (1+\|u\|_{L^p(\Omega;\cH)}^p)$.
\end{itemize}
\label{ass:psilip}
\end{assumption}

The covariance matrix of random variables $Z, X \in \cH$ 
will be denoted
\[
\mathrm{Cov}[Z,X] \coloneq\E{ (Z-\E{Z}) \otimes (X-\E{X} ) },
\]
with the shorthand  $\mathrm{Cov}[Z] \coloneq \mathrm{Cov}[Z,Z]$.

\subsection{Some details on Hilbert spaces, Hilbert-Schmidt operators, and Cameron-Martin spaces}
\label{ssec:dets}

Let $\cK_1$ and $\cK_2$ be two separable Hilbert \hh{spaces} with inner products $\langle \cdot, \cdot \rangle_{\cH}$ and $\langle \cdot, \cdot \rangle_{\cK}$ and the induced norms
\begin{equation}
\|u\|_{\cK_1} := \langle u, u \rangle_{\cK_1}^{1/2}, \quad  \text{and} \quad
\|u\|_{\cK_2} := \langle u, u \rangle_{\cK_2}^{1/2}.
\end{equation}
The tensor product of $\cK_1$ and $\cK_2$ is a Hilbert space with the inner product defined by
\begin{equation}
 \langle u \otimes v, u' \otimes v' \rangle_{\cK_1 \otimes \cK_2} = \langle u, u' \rangle_{\cK_1}\langle v, v'  \rangle_{\cK_2} \qquad \forall u, u' \in \cK_1, \quad \forall v,v' \in \cK_2
\end{equation}
and extended by linearity to finite sums.
The tensor product $\cK_1 \otimes \cK_2$ is the completion of this set with respect to the induced norm $\|\cdot\|_{\cK_1 \otimes \cK_2}$. 
It holds that
\begin{equation}\label{eq:hsdef}
\|u \otimes v\|_{\cK_1 \otimes \cK_2} = \|u\|_{\cK_1}\|v\|_{\cK_2}.
\end{equation}
Notice furthermore that every $u \otimes v \in \cK_1 \otimes \cK_2$ can be identified with a bounded linear mapping 
\begin{equation}\label{Tuv-def}
 T_{u,v}: \cK_2^* \to \cK_1 \quad \text{with} \quad T_{u,v}(f) \coloneq f(v) u, \text{ for} f \in \cK_2^*.
\end{equation}
For two bounded linear operators $A,B: \cK_2^* \to \cK_1$ we recall the definition of the Hilbert-Schmidt inner product and the norm
\begin{equation}
 \langle A,B \rangle_{HS} = \sum_{k=1}^\infty \langle Ae^*_k, Be^*_k \rangle_{\cK_1},
\qquad 
|A|_{HS} = \langle A,A \rangle_{HS}^{1/2},
\end{equation}
where $\{e^*_k\}_{k=1}^\infty$ is any complete orthonormal sequence in
$\cK_2^*$. A bounded linear operator $A:\cK_2^* \to \cK_1$ is called a
Hilbert-Schmidt operator if $|A|_{HS} < \infty$ and $HS(\cK_2^*,
\cK_1)$ is the space of all such operators. In view of \eqref{Tuv-def}
we observe
\begin{multline*}
 |T_{u,v}|_{HS}^2 = \sum_{k=1}^\infty \langle e^*_k(v)u, e^*_k(v) u \rangle_{\cK_1} 
\\= \|u\|_{\cK_2}^2\sum_{k=1}^\infty |e^*_k(v)|^2 
= \|u\|_{\cK_1}^2\|v\|_{\cK_2}^2 = \|u \otimes v\|_{\cK_1 \otimes \cK_2},
\end{multline*}
and therefore the tensor product space $\cK_1 \otimes \cK_2$ is isometrically isomorphic to $HS(\cK_2^*,\cK_1)$ (and to $HS(\cK_2,\cK_1)$ by the Riesz representation theorem). For an element $A \in \cK_1\otimes \cK_2$ we identify the norms
\begin{equation}
 \|A\|_{\cK_1 \otimes \cK_2} = |A|_{HS}.
\end{equation}

Consider the Gaussian random variable $u \sim \mu_0 := N(0,C)$.  
Provided the spectrum of $C$ is trace-class, then it has  
an eigen-basis which is orthonormal with respect to $\cH$, in the
sense that $C\phi_k=\lambda_k\phi_k$, $\langle
\phi_j,\phi_k\rangle_\cH = \delta_{j,k}$, and $\sum_{k=0}^\infty
\lambda_k < \infty$.  It is easy to see that $u\in \cH$
$\mu_0$-almost surely.  
The space $E:=\{v\in\cH; \| C^{-1/2} v\|_\cH<\infty\}$ 
is known as the Cameron-Martin space, and it is also
clear, by Kolmogorov's three series theorem, cf.~\cite{chung2001},
that $u\sim \mu_0 \Rightarrow u \notin E$ almost surely. In fact,
$E\subset \cH \subset E^*$, where $E^*$ denotes the dual of $E$ wrt
the inner product $\langle\cdot,\cdot\rangle_\cH$, and $C:E^*
\rightarrow E$.  

\begin{proposition}\label{prop:fincov}
If $u\in L^2(\Omega;\cH)$ then  
$C:=\E{(u-\E{u}) \otimes (u - \E{u})} \in \cH \otimes \cH$.
Furthermore, $C: \cH \rightarrow E^2$, where 
$E^2 := \{v\in\cH; \| C^{-1} v\|_\cH<\infty\} \subset E$.
\end{proposition}
\begin{proof}
 
{Notice that $\| \E{u} \|^2_{\cH} \leq \E{ \| u \|^2_{\cH}}$ by \hh{Jensen's} inequality, 
so $\E{ u}\in \cH$, since $u\in L^2(\Omega;\cH)$.
Without loss of generality let $\E{ u}=0$.
Noting that Tr$(\E{ u\otimes u}) = \E{\|u\|^2_{\cH}}$ provides the first claim.
The second part is obvious since  $v = C^{-1} ( C v)$.}
\end{proof}

\subsection{EnKF}
\label{subsec:enkf}

EnKF uses an ensemble of particles to estimate means and covariance matrices appearing in the Kalman
filter, however the framework can be generalized to non-Gaussian models. 
Let $v_{n,i}$, $\hv_{n,i}$ 
respectively denote the prediction and update of the $i$-th particle at simulation time $n$.  
One EnKF two-step transition consists 
of the propagation of an ensemble $\{\hv_{n,i}\}_{i=1}^M  \mapsto \{\hv_{n+1,i}\}_{i=1}^M$.
\footnote{Due to the implicit linear and Gaussian assumptions underlying the formulation, 
one may determine that it is reasonable to summarize the ensemble by its sample mean and 
covariance and indeed this is often done.  In this case, one may  construct a Gaussian from 
the empirical statistics and resample from that.}
This procedure consists nonetheless in the predict and update steps. 
In the predict step, 
$M$ particle paths are computed over one interval, i.e., 
 \begin{equation}
  v_{n+1}(\omega_i) = \Psi(\hv_{n}(\omega_i),\omega_i) 
   \end{equation}
   for $i=1, \ldots, M$, where $v_{n} (\omega_i) := v_{n,i}$ denotes a
   realization corresponding to the event sample $\omega_i$ of the
   random variable $v_{n} : \Omega \rightarrow \cH$, and 
$\Psi(\cdot,\omega_i)$ signifies the corresponding realization of
   the map for a given initial condition.  Indeed the notation for
   random variable realizations, e.g. $\xi_{n,i}$ and
   $\xi_{n}(\omega_i)$, will be used interchangeably where confusion
   is not possible.  The impetus for introduction of the latter
   notation will become apparent in the next section.  For this
   presentation it suffices to assume a single infinite precision map,
   however there indeed may also be numerical approximation errors,
   i.e. $\Psi^L$ may be used in place of $\Psi$ for some satisfactory
   resolution $L$.  The prediction step is completed by using the
   particle paths to compute sample mean and covariance operator:
  \[
  \begin{split}
  m^{\rm MC}_{n+1}  		& = E_M[v_{n+1}] \ ,\\
  C_{n+1}^{\rm MC}   		& = \cov_M[v_{n+1}] \ ,
  \end{split}
  \]
  with the unbiased sample moments
\begin{equation}\label{eq:sampleAvg}
E_M[v] \coloneq \frac{1}{M} \sum_{i=1}^M  v(\omega_{i} ) \ , 
\end{equation}
and
\begin{equation}\label{eq:sampleCov}
\cov_M[u,v] \coloneq \frac{M}{M-1}\parenthesis{E_M[u \otimes v] - E_M[u] \otimes E_M[v]},
\end{equation}
and the shorthand $\cov_M[u] \coloneq \cov_M[u,u]$.
The update step consists of computing (1) auxiliary operators 
\begin{equation}
 S^{\rm MC}_{n+1}            = HC_{n+1}^{\rm MC}  H^* + \Gamma \text{ and }   K^{\rm MC}_{n+1}   = \left( C_{n+1}^{\rm MC} H^* \right) (S^{\rm MC}_{n+1})^{-1}, 
\label{eq:auxop1}
\end{equation}
where $H^*$ is the adjoint of $H$ defined by 
$\langle a,Hu \rangle_{{\cR_m}} = \langle H^* a, u\rangle_\cH$ for all $a\in\bbR^m$ and $u \in \cH$,
and (2) measurement corrected particle paths for $i=1,2, \ldots, M$,
 \[ 
 \begin{split}
  \yTilde{n+1,i}	 & = y_{n+1} + \eta_{n+1,i},\\
  \hat{v}_{n+1,i} 
  & = (I - K^{\rm MC}_{n+1}H) v_{n+1,i}
  + K^{\rm MC}_{n+1} \yTilde{n+1,i}, 
\end{split}
\]
where the sequence $\{ \eta_{n+1,i} \}_{i=1}^M$ is i.i.d.~with $\eta_{n+1,1} \sim N(0, \Gamma)$. 
This last procedure may appear somewhat ad-hoc.  Indeed it was originally introduced 
in \cite{burgers1998analysis} to correct the statistical error induced in its absence in implementations 
following the original formulation of the ensemble Kalman filter in \cite{evensen1994sequential}.
It has become known as the perturbed observation implementation.
Due to the form of the update, all ensemble members are correlated to one another after 
the first update.  So, even in the linear Gaussian case, 
the ensemble is no 
longer Gaussian after the first update.  Nonetheless, it has been shown that the limiting ensemble
converges to the correct Gaussian in the linear and finite-dimensional case \cite{mandel2011convergence, le2011large},
with the rate $\cO(N^{-1/2})$ in $L^p$ for Lipschitz functionals with polynomial growth at infinity.  
Furthermore, it converges with the same
rate in the nonlinear but Lipschitz case, i.e. under Assumption~\ref{ass:psilip} 
\cite{le2011large, law2014deterministic},
to a limiting distribution which will be discussed further in the subsection \ref{sec:mfenkf}.  
The measurement corrected sample mean and covariance, which need not be computed,
would be given by:
\[
  \begin{split}
  \hm_{n+1}^{\rm MC}  		& = E_M[\hat{v}_{n+1}],\\
  \hc_{n+1}^{\rm MC}   		& = \cov_M[\hat{v}_{n+1}]. 
  \end{split}
\]


For later computing quantities of interest, we introduce the following
notation for the empirical measure of the EnKF ensemble
$\{\hv_{n,i}\}_{i=1}^{M}$:\footnote{Similar may be done for the
  predicting distributions, but the updated distributions will be our
  primary interest.}
\begin{equation}
\hat \mu^{\rm MC}_n = \frac{1}{M} \sum_{i=1}^{M_0} \delta_{\hv_{n,i}}.
\label{eq:emp}
\end{equation}
And for any 
$\varphi: \cH \rightarrow \bbR$, let 
\[
\hat \mu_n^{\rm MC}(\varphi) := \int \varphi d\mu^{\rm ML}_n = 
 \frac{1}{M}\sum_{i=1}^{M} \varphi(\hv_{n,i}) .
\]

This section is concluded with a comment regarding the required computation of auxiliary operators
\eqref{eq:auxop1}.  
In particular, it will be convenient to introduce index summation notation so that it is assumed that indices 
which appear twice will be summed over, i.e. $a_kb_k := \sum_{k} a_k b_k$.
Letting $\{e_i\}_{i=1}^m$ be a basis for $\bbR^m$, one can write $H = 
H_{ik} e_i \otimes \phi_k$, where
$H_{ik} \coloneq \langle e_i , H \phi_k \rangle$,
and $\covHatMC{n+1} = 
C_{n+1,kl}^{{\rm MC}} \phi_k \otimes \phi_l$, where
$C_{n+1,kl}^{{\rm MC}} \coloneq \langle \phi_k,\covHatMC{n+1} \phi_l\rangle$.
Then it makes sense to define the intermediate operator 
\begin{equation}
R^{\rm MC}_{n+1} = 
R^{{\rm MC}}_{n+1,ki} \phi_k \otimes e_i,
\label{eq:arrr}
\end{equation}
where $R^{{\rm MC}}_{n+1,ki}= C_{n+1,kl}^{{\rm MC}} H_{il}$.
The operators of \eqref{eq:auxop1} can be written in terms of indices as 
\begin{equation}
 S^{\rm MC}_{n+1,ij} = H_{il} R^{{\rm MC}}_{n+1,lj} + \Gamma_{ij} \text{ and }   
 K^{{\rm MC}}_{n+1,ki}   = R^{{\rm MC}}_{n+1,kg} \left((S^{\rm MC}_{n+1})^{-1}\right)_{gi}, 
\label{eq:auxop2}
\end{equation}
where the ranges of the indices $k,l =1,2,\dots$ and $i,j,g = 1,2,\dots, m$ are understood.

\subsection{Multilevel EnKF}
\label{subsec:mlenkf}

Herein a hierarchy of spaces are introduced $\cH_\ell = {\rm span} \{\phi_l\}_{l=1}^{N_\ell}$,
{where $\{N_\ell\}$ is an exponentially increasing sequence of natural numbers further
  described in Assumption~\ref{ass:mlrates}}.
Define $\Phi_\ell = [\phi_1,\dots, \phi_{{N_\ell}}] : \bbR^{N_\ell} \rightarrow \cH$ and 
the projection operator $\cP_\ell := \Phi_\ell\Phi_\ell^\top$. 
For $u \in \cH$,
$u^\ell = \cP_\ell u = \sum_{l=1}^{N_\ell} u_{l} \phi_l \in \cH_\ell$,
where $u_{l}=\langle \phi_l, u\rangle$.   
One has that $\cH \supset \dots \supset \cH_{\ell+1} \supset \cH_\ell \supset \dots \supset \cH_0$.
MLEnKF computes particle paths on this 
hierarchy of spaces with a hierarchy of accuracy levels.  
The case where the accuracy levels are given by refinement of the temporal discretization 
\hh{has already been} covered in \cite{ourmlenkf}, for \hh{finite-dimensional} state space.  
Let $v^{\ell}_{n}$, $\hv^\ell_n$ respectively denote the prediction 
and update 
of a particle on solution level $\ell$ at simulation time $n$. 
A solution on level $\ell$ is computed by the numerical integrator $v^{\ell}_{n+1} = \psiL(\hv^{\ell}_{n})$.
Furthermore, let the increment operator for level $\ell$ be given by 
\[
\Delta_\ell v_n   := 
\begin{cases} 
v_n^0, & \text{if } \ell =0,\\
v^\ell_n  - v^{\ell-1}_n , & \text{else if } \ell >0.
\end{cases}
\]
Then the transition from approximation of the distribution of $u_n|Y_n$ to 
the distribution of $u_{n+1}| Y_{n+1}$ in the MLEnKF framework 
consists of the predict/update step of generating {\it pairwise coupled} particle realizations 
on a set of levels $\ell=0,1,\ldots, L$.  However, it is important to note that here one has 
correlation between pairs and also between levels due to the update, 
unlike the standard MLMC in which one has i.i.d. pairs.  This point will be very important,
and we return to it in the following section.   

Similarly to the standard EnKF, the MLEnKF transition is between {\it multilevel} 
ensembles 
$\{(\hv^{\ell}_{n,i})_{i=1}^{M_\ell}\}_{\ell=1}^L \mapsto 
\{(\hv^{\ell}_{n+1,i})_{i=1}^{M_\ell}\}_{\ell=1}^L$.
This consists, as for EnKF, of the predict and update steps. 
In the predict step, particle paths are first computed on a hierarchy of levels. That is, 
the particle paths are computed one step forward by
\begin{equation}\label{eq:DlvHatDef}
\begin{split}
  v^{\ell-1}_{n+1}(\omegaLI)  	& = \Psi^{\ell-1}(\hv^{\ell-1}_{n}(\omegaLI),\omegaLI),\\
  v^{\ell}_{n+1}(\omegaLI) 	 	& = \Psi^\ell(\hv^{\ell}_{n}(\omegaLI), \omegaLI),
  \end{split}
\end{equation}
for the levels $\ell =0,1,\ldots,L$ and level particles $i = 1,2, \ldots, M_\ell$
(where for convenience we introduce the convention that $\vHatL{-1}{} :=0$). 
Here the 
noise in the second argument of the $\psiL$ is 
correlated only within 
pairs, and are otherwise independent. 
Thereafter, sample mean and covariance matrices are computed
as a sum of sample moments of increments over all levels:
\begin{equation}
\label{eq:mlSampleMoments}
\begin{split}
\mml_{n+1} &= \sum_{\ell=0}^L E_{M_\ell}[\Delta_\ell v_{n+1}(\omega_{\ell,\cdot})],\\
\cml_{n+1}  &= \sum_{\ell=0}^L \cov_{M_\ell}[ v^{\ell}_{n+1}(\omega_{\ell,\cdot} )]
- \cov_{M_\ell}[ v^{\ell-1}_{n+1}(\omega_{\ell,\cdot} )],
\end{split}
\end{equation}
where we recall the sample moment notation~\eqref{eq:sampleAvg} and~\eqref{eq:sampleCov}.
Define 
\begin{equation}\label{eq:ex}
X_{M_\ell} := \frac1{\sqrt{M_{\ell}-1}}\left([v^{\ell}_{n+1}(\omega_{\ell,1}), \dots, v^{\ell}_{n+1}(\omega_{\ell,M_\ell})] - 
{E_{M_\ell}[v^{\ell}_{n+1}(\omega_{\ell,\cdot})] {\bf 1}^\top}\right),
\end{equation}
where ${\bf 1}$ is a vector of $M_\ell$ ones. 
Then $\cov_{M_\ell} [ v^{\ell}_{n+1}(\omega_{\ell,\cdot} ) ] = 
X_{M_\ell} X_{M_\ell}^\top$.
The cost of construction is $N_\ell^2 \times M_\ell$, and would therefore be the dominant 
level $\ell$ cost.  It turns out 
{it is not necessary to construct the full covariance}, as will be described below.

Recalling \eqref{eq:auxop1} and \eqref{eq:auxop2}, 
it is necessary for \hh{the} stability of the algorithm that the 
matrix $H R^{\rm ML}_n$ 
appearing in the denominator of the gain (where $R^{\rm ML}_n$ is the multilevel version of the 
operator defined in the Monte Carlo context in equation \eqref{eq:arrr}) is positive semi-definite, a condition
which is {\it not} guaranteed for multilevel estimators.  
This will therefore be {\it imposed} in the algorithm, similarly to the strategy in 
the recent work \cite{ourmlenkf}.  
Let 
$$
H R^{\rm ML}_n = \sum_{i=1}^m \lambda_i q_i q_i^\transpose
$$
denote the eigenvalue decomposition of $H R^{\rm ML}_n$. 
Notice that the multilevel covariance does not ensure min$_i(\lambda_i) \ngeq 0$.  Define
\begin{equation}
H \tilde{R}^{\rm ML}_n = \sum_{i=1; \lambda_i\geq 0}^m \lambda_i q_i q_i^\transpose.
\label{eq:covzee}
\end{equation}
In the update 
step 
the multilevel Kalman gain is 
defined as follows
\begin{equation}
 \kml{n+1}  = R^{\rm ML}_{n+1} 
 (S^{\rm ML}_{n+1})^{-1}, \text{ where } 
  S^{\rm ML}_{n+1}  := H \tilde{R
  }^{\rm ML}_{n+1}  
  + \Gamma. 
\label{eq:newkay}
\end{equation}
{Next, all particle paths are corrected according to measurements and perturbed observations are 
added:
\begin{equation}\label{eq:upsamps}
\begin{split}
  \yTildeL{\ell}{n+1,i}		& = y_{n+1} + \etaL_{n+1,i} \\
  \hv^{\ell-1}_{n+1}(\omega_{i,\ell})  & =     (I - \cP_{\ell-1} \kml{n+1} H ) \vHatL{\ell-1}{n+1}(\omega_{i,\ell} ) + 
  \cP_{\ell-1} \kml{n+1} \yTildeL{\ell}{n+1,i},  \\
  \hv^{\ell}_{n+1}(\omega_{i,\ell})  & =     (I - \cP_\ell \kml{n+1} H ) \vHatL{\ell}{n+1}(\omega_{i,\ell} ) + 
  \cP_{\ell} \kml{n+1} \yTildeL{\ell}{n+1,i},  
  \end{split}
\end{equation}
where the sequence $\{\etaL_{n+1,i}\}_{i=1}^N$ is i.i.d.~with
$\eta^{\{0\}}_{n+1,1} \sim N(0, \Gamma)$.}  It is in this step
precisely that the pairs all become correlated with one another and
the situation becomes significantly more complex than the i.i.d. case.
After the first update, this correlation propagates forward through
\eqref{eq:DlvHatDef} to the next observation time via this ensemble.
This is the conclusion of the update step of the MLEnKF, and this
multilevel ensemble is subsequently propagated forward to the next
prediction 
time via \eqref{eq:DlvHatDef}.
	
\begin{proposition}\label{prop:covcost}
Assuming $m \ll N_0$, 
the cost arising from level $\ell$ in the construction of the $M_\ell$ sample 
updates \eqref{eq:upsamps} is proportional to 
$m \times N_\ell \times M_\ell$.
\end{proposition}	
	
\begin{proof}
Two separate operations are required at each level $\ell$.  
The first arises in the construction of the multilevel gain $\kml{n+1}$ 
in \eqref{eq:newkay}.  Now shall become apparent the impetus for 
introducing the operator $R^{\rm ML}_{n+1} = C^{\rm ML}_{n+1}H^*$ 
in \eqref{eq:arrr}.  Notice at no point is the full $C^{\rm ML}_{n+1}$
required, but rather only 
$$
R^{\rm ML}_{n+1} = \sum_{\ell=0}^L \cov_{M_\ell}[ v^{\ell}_{n+1}(\omega_{\ell,\cdot} ), 
H v^{\ell}_{n+1}(\omega_{\ell,\cdot} )]
- \cov_{M_\ell}[ v^{\ell-1}_{n+1}(\omega_{\ell,\cdot} ), 
H v^{\ell-1}_{n+1}(\omega_{\ell,\cdot} )].
$$
The level $\ell$ contribution to this is dominated by the operation
$X_{M_\ell} (H X_{M_\ell})^\top$, where $X_{M_\ell}$ is defined in \eqref{eq:ex}.
The cost of constructing $H X_{M_\ell} \in \bbR^{m\times M_\ell}$ 
is proportional to $m \times N_\ell \times M_\ell$, and so the cost of constructing 
$X_{M_\ell} (H X_{M_\ell})^\top$ is proportional to 
$2 \times m \times N_\ell \times M_\ell$.  
There is also an insignificant one time cost of $\cO(m^2N_L)$ in the construction 
and inversion of $S^{\rm ML}_{n+1}$.

The second operation at level $\ell$ arises from actually computing
the update \eqref{eq:upsamps} using $\cP_\ell \kml{n+1}$.
The cost of obtaining $\cP_\ell \kml{n+1}$ from $\kml{n+1}$
is negligible, so it is clear that 
each sample incurs a cost $m \times N_\ell$.
\end{proof}

The following notation denotes the empirical measure of the multilevel
ensemble $\{(\hv^{\ell}_{n,i})_{i=1}^{M_\ell}\}_{\ell=1}^L$:
\begin{equation}
\hat \mu^{\rm ML}_n = \frac{1}{M_0} \sum_{i=1}^{M_0} \delta_{
  \hv^{0}_{n}(\omega_{i,0})} + 
\sum_{\ell=1}^{L} \frac{1}{M_\ell} 
\sum_{i=1}^{M_\ell}( \delta_{\hv^{\ell}_{n}(\omega_{i,\ell}) } -
\delta_{\hv^{\ell-1}_{n}(\omega_{i,\ell})} ),
\label{eq:mlemp}
\end{equation}
and for any $\varphi: \cH \rightarrow \bbR$, let 
\[
\hat \mu_n^{\rm ML}(\varphi) := \int \varphi d\mu^{\rm ML}_n = 
\sum_{\ell=0}^{L}  \frac{1}{M_\ell}\sum_{i=1}^{M_\ell} 
{\varphi(\hv^{\ell}_{n}(\omega_{i,\ell})) - \varphi(\hv^{\ell-1}_{n}(\omega_{i,\ell}))}.
\]

\subsection{Nonlinear Kalman filtering}
\label{sec:mfenkf}

It will be useful to introduce the limiting process, 
in the case of nonlinear non-Gaussian forward model \eqref{eq:psiDefinition}.
The following process defines the MFEnKF \cite{law2014deterministic}:
\begin{equation}
\;\;\;\;\;\;\;\mbox{Prediction}\;\left\{\begin{array}{lll}
    \mfv_{n+1}& = \Psi (\hmfv_n), \\
\mfm_{n+1}&=\E{\mfv_{n+1}},\\
\mfc_{n+1}&=\E{({\mfv}_{n+1}-\mfm_{n+1})\otimes({\mfv}_{n+1}-\mfm_{n+1})}
 \end{array}\right.
\label{eq:mfpredict}
\end{equation}
\begin{equation}
\mbox{Update}\left\{\begin{array}{llll} \bar S_{n+1}&=(H\mfc_{n+1})H^* +\Gamma \\
\mfk_{n+1}&=(\mfc_{n+1}H^*) \bar S_{n+1}^{-1}\\
{\tilde y}_{n+1}&=y_{n+1}+\eta_{n+1}\\
\hmfv_{n+1}&=(I- \mfk_{n+1}H){\mfv}_{n+1}+\mfk_{n+1}{\tilde y}_{n+1}.\\
\end{array}\right.
\label{eq:mfupdate}
\end{equation}
\vspace{4pt}
Here $\eta_n$ are i.i.d. draws from 
$N(0,\Gamma).$  
It is easy to see that in the linear Gaussian case 
the mean and variance of the 
above process correspond to the mean and variance of the filtering distribution 
\cite{law2015data}.
Moreover, it was shown in \cite{mandel2011convergence, le2011large} that 
for finite-dimensional state-space the single level EnKF converges to the Kalman 
filtering distribution with the standard rate $\cO(M^{-1/2})$ in this case. 
It was furthermore shown in \cite{le2011large} and \cite{law2014deterministic} 
that for nonlinear Gaussian state-space models and fully non-Gaussian models 
\eqref{eq:psiDefinition}, 
respectively, the EnKF converges to the above process with the same rate 
as long as the models satisfy a Lipschitz criterion as in Assumption~\ref{ass:psilip}.
The work of \cite{ourmlenkf} illustrated that the MLEnKF converges as well, and with 
an asymptotic cost-to-$\varepsilon$ which is strictly smaller than its single level 
EnKF counterpart.  The work of \cite{kwiatkowski2014convergence} 
extended convergence results to \hh{infinite-dimensional} state-space for square root filters.
In this work, the aim is to prove convergence of the MLEnKF 
for infinite-dimensional state-space, with the same favorable asymptotic cost-to-$\varepsilon$ \hh{peformance}. 

The following fact will be necessary in the subsequent section.

\begin{proposition}\label{prop:reg}
Given Assumption \ref{ass:psilip} on $\Psi$, 
the MFEnKF process \eqref{eq:mfpredict}--\eqref{eq:mfupdate} satisfies 
$\bar{v}_n, \hat{\bar{v}}_n \in L^p(\Omega;\cH)$ for all {$n \in \bbN$}.
\end{proposition}

\begin{proof}
Clearly it holds for time $n=0$.  
Given $\hat{\bar{v}}_n\in L^p(\Omega;\cH)$, Assumption \ref{ass:psilip} (ii)
guarantees $\bar{v}_{n+1} \in L^p(\Omega;\cH)$.
By Proposition \ref{prop:fincov}, $\mfc_{n+1} \in \cH\otimes \cH$.
Since $H\mfc_{n+1}H^* \geq 0$ and $\Gamma>0$, it is clear that 
$\bar S_{n+1}>0$, which implies 
$\| H^* \bar S_{n+1}^{-1} \|_{\cH \otimes \cR_m}<\infty$.
Hence, $\mfk_{n+1} \in \cH \otimes \cR_m$.  
Therefore it is clear that $\hat{\bar{v}}_{n+1} \in L^p(\Omega;\cH)$.
\end{proof}

\section{Theoretical Results}
\label{sec:theory}

The approximation error and
computational cost of approximating the true 
filtering distribution by MLEnKF when given a sequence of 
observations $y_1, y_2, \ldots, y_n$ 
will be studied in this section. 
Before stating the main approximation theorem, it will be useful to 
present the basic assumptions that will be used throughout and the 
corresponding standard MLMC approximation results for i.i.d. samples,
as well as a slight variant which will be useful in what follows.

\begin{definition}\label{def:globalLip} {A function $\varphi: \cH \to
    \bbR$ is said to be globally Lipschitz continuous provided there
    exist a positive scalar $C_\varphi <\infty$ {such that for all $u,v\in \cH$}
\begin{equation}\label{eq:localLipschitz}
\abs{\varphi(u) -\varphi(v)} \leq C_\varphi \hNorm{u-v}.
\end{equation}}
\end{definition}

\begin{assumption} 
\label{ass:mlrates}
Consider the hidden Markov model defined by \eqref{eq:psiDefinition}
and \eqref{eq:obdef} with initial data $u_0 \in L^p(\Omega;\cH)$ for all
{$p \geq 2$} and assume that the sequence of resolution dimensions $\{N_\ell\}$
fulfils the exponential growth constraint $N_\ell \eqsim \kappa^\ell$,
for some $\kappa >1$.
Let $\psiL$ denote a numerical solver with a resolution parameter
{$h_\ell \eqsim N_\ell^{-1/d}$.}
This will define the hierarchy of solution operators in 
Section~\ref{sec:kalman}, which are assumed to 
satisfy Assumption \ref{ass:psilip}. 
 {For} a given set of constants $\beta,
\gamma>0$, assume the following conditions are fulfilled 
for all $\ell \geq 0$ and $u,v \in L^p(\Omega; \cH)$ for all $p\ge 2$:
\begin{enumerate}

\item[(i)] $\| \psiL(u) - \Psi(u) \|_{L^p(\Omega; \cH)} 
\lesssim 
h_\ell^{\beta/2}$,
\text{ for all }$p\ge 2$, 

\item[(ii)] $\|(I- \cP_\ell) u_0 \|_{L^p(\Omega; \cH)} \lesssim h_\ell^{\beta/2}$, {\text{ for all }$p\ge 2$,}


\item[(iii)] 
  $\costL  \lesssim h_\ell^{-d \gamma}$, where $\costL$ denotes the computational cost associated to level $\ell$
  (and $d$ is the spatiotemporal dimension of the continuum which is being approximated)\footnote{This
  can be made much more general, but the objective here is simplicity
  of exposition}.

 \end{enumerate}
\end{assumption}

Assumption~\ref{ass:mlrates} is given in a bare-minimum form, which we
believe will be easier to verify when applying the method to particular
problems. The next corollary states direct consequences of the above
assumption, which will be useful for proving properties of the MLEnKF
method.

\begin{proposition}\label{prop:asimp}
  Suppose Assumption~\ref{ass:mlrates} holds and $\Psi^\ell = \cP_\ell\Psi$.
  {Then for all $\ell \in \bbN$, $u,v \in  L^p(\Omega; \cH)$ for all $p\ge 2$,
  and globally Lipschitz continuous observables $\varphi$}:
\begin{itemize}
\item[(i)] 
$\| \psiL(v) - \Psi^{\ell-1}(v) \|_{L^p(\Omega; \cH)} 
\lesssim h_\ell^{\beta/2}$,
for all $p\ge 2$, 
\item[(ii)] $\abs{ \E{\varphi(\psiL(u)) - \varphi(\Psi(v)) }} \lesssim \|u-v\|_{L^p(\Omega; \cH)} + h_\ell^{\beta/2}$,
{ for all $p \ge 2$},
\item[(iii)] $\|(I- \cP_\ell) \bar C_n \|_{\cH\otimes \cH} \lesssim h_\ell^{\beta/2}$.
\end{itemize}
\end{proposition}

\begin{proof}
Property (i) follows from Assumption \ref{ass:mlrates}(i) and Minkowski's inequality.  
{Property (ii) follows from Definition~\ref{def:globalLip}, 
followed by Minkowski's inequality, Assumption \ref{ass:psilip}(i), and Assumption \ref{ass:mlrates}(i).}
For property (iii), recall Proposition \ref{prop:reg}, 
and without loss of generality assume {$\E{ \bar{v}_n}=0$} 
(for simplicity of the argument to follow).  
Now observe
\[
\|(I-\cP_\ell) \bar{C}_n\|_{\cH\otimes \cH} = 
\|\bbE[(I-\cP_\ell) \bar{v}_n \otimes \bar{v}_n] \|_{\cH\otimes \cH} 
\leq \|(I-\cP_\ell) \bar{v}_n\|_2 \| \bar{v}_n \|_2,
\] 
where the inequality is a result of Jensen's inequality, the definition \eqref{eq:hsdef},
and H\"older's inequality.  Notice that 
$(I-\cP_\ell) \bar{v}_n = (I-\cP_\ell)\Psi(\hat{\bar{v}}_{n-1})$.
Since it is assumed that $\Psi^\ell = \cP_\ell\Psi$, the claim follows
from Assumption \ref{ass:mlrates}(i) again.
\end{proof}


\begin{remark}
It will be assumed that the computational cost of the forward
simulation, Cost$(\Psi^\ell) =\cO(h_\ell^{-d\gamma})$ is at least
linear in {$N_\ell$, i.e., that $\gamma \ge 1$}, and that $m\ll N_0$.  Therefore, in
view of Proposition \ref{prop:covcost}, the total cost is dominated by
$\cC_\ell=\cO(h_\ell^{-d\gamma})$. It is important to observe that in
the big data case $m {\ge} N_0$, the algorithm
  will need to be modified to be efficient in the non-asymptotic regime when
the accuracy constraint $\varepsilon$, relatively speaking, is large.
{For larger values of $m$, 
smaller $\varepsilon$ regimes will be affected.}
\end{remark}

We will now state the main theorem of this paper. It gives an upper
bound for the computational cost of achieving a sought accuracy in
$L^p(\Omega; \cH)$-norm when using the MLEnKF method to approximate
the expectation of an observable.  The theorem may be considered an
extension to spatially extended models of the earlier work
\cite{ourmlenkf}.  

\begin{theorem}[MLEnKF accuracy vs. cost]\label{thm:main}
{Consider a globally Lipschitz continuous observable function
$\varphi:\cH \to \bbR$, and suppose
Assumptions~\ref{ass:psilip} and~\ref{ass:mlrates} hold. For a given
{$\varepsilon >0$}, let $L$ and $\{M_\ell\}_{\ell=0}^L$ be defined under the
constraints $L = \lceil 2\log_\kappa(\varepsilon^{-1})/\beta\rceil$ and}
\begin{equation}\label{eq:chooseMlr}
{M_\ell \eqsim  
 \begin{cases} 
    h_\ell^{(\beta+d\gamma)/2}  h^{-\beta}_L, & \text{if} \quad \beta > d\gamma, \\
   h_\ell^{(\beta+d\gamma)/2}  L^2 h^{-\beta}_L, & \text{if} \quad  \beta = d\gamma, \\
   h_\ell^{(\beta + d\gamma)/2} h^{-(\beta + d\gamma)/2}_L, &\text{if} \quad \beta< d\gamma. 
 \end{cases}}
\end{equation}
Then, {for any $p\ge 2$ and $n \in \bbN$},
\begin{equation}
\|\hat \mu^{\rm ML}_n (\varphi) - \hat \mfmu_n (\varphi) \|_{L^p(\Omega;\cH)} \lesssim |\log(\varepsilon)|^n\varepsilon,
\label{eq:lperror}
\end{equation}
where $\hat \mu^{\rm ML}_n$ denotes the multilevel empirical measure defined in
\eqref{eq:mlemp} whose particle evolution is given by the multilevel predict
\eqref{eq:DlvHatDef} and update \eqref{eq:upsamps} formulae,
approximating the time $n$ mean-field EnKF distribution $\hat \mfmu_n$
 (the filtering distribution $\hat \mfmu_n=N(\hm_n, \hc_n)$ in the linear Gaussian case).

The computational cost of the MLEnKF estimator over the time
sequence is bounded by 
\begin{equation}\label{eq:mlenkfCosts2}
\cost{\mathrm{MLEnKF}} \lesssim 
\begin{cases}
\varepsilon^{-2}, & \text{if} \quad \beta > d\gamma,\\              
\varepsilon^{-2} \abs{\log(\varepsilon)}^3, & \text{if} \quad \beta = d\gamma,\\
\varepsilon^{- 2d\gamma/\beta}, & \text{if} \quad \beta < d\gamma.
\end{cases}
\end{equation}
\end{theorem}

\hh{The proof follows very closely that of~\cite[Theorem 3.2]{ourmlenkf}}, except
here it is extended to the Hilbert space setting with approximation of
spatially extended models.


Following \cite{ourmlenkf} and \cite{le2011large, law2014deterministic, mandel2011convergence}, 
introduce the mean-field limiting multilevel ensemble
$\{(\mfv^{\ell}_{n,i})_{i=1}^{M_\ell}\}_{\ell=1}^L$, which
evolves according to the same equations {\it with the same
  realizations of noise} except the covariance $\mfc_n$, \hh{and} hence the
Kalman gain $\mfk_n$, are given by limiting formulae
in~\eqref{eq:mfpredict} and~\eqref{eq:mfupdate}.  
An ensemble member $\mfv^{\ell}$ corresponds to a solution of the
above system with
$\mfv^\ell_{n+1} = \psiL (\hmfv^\ell_n)$ replacing the
first equation and the equation
\[
\hmfv^\ell_{n+1}=(I-\cP_\ell \mfk_{n+1}H) \mfv^\ell_{n+1}+\cP_\ell \mfk_{n+1}{\tilde y}_{n+1}^{\ell}
\]
replacing the last equation.  The sample $\mfv^{\ell}(\omegaLI)$ is a
single realization of this system above with the same noise
realization $\omegaLI$ as the sample $v^{\ell}(\omegaLI)$ from MLEnKF,
including the perturbed observation.  
Note that the processes
$\mfv^{\ell}$, $\hmfv^{\ell}$ are bounded in $L^p(\Omega; \cH)$ 
as well by similar arguments to Proposition \ref{prop:reg}.


Let us first recall that the multilevel Kalman gain is defined by
\begin{equation}
K^{\rm ML}_n = R^{\rm ML}_n 
(H \tilde{R}^{\rm ML}_n 
+ \Gamma)^{-1}
\nonumber
\end{equation}
where 
\begin{equation}
H \tilde{R}^{\rm ML}_n = \sum_{i=1; \lambda_i > 0}^m \lambda_i q_i q_i^\transpose,
\label{eq:covzee1}
\end{equation}
for eigenpairs $\{\lambda_i,q_i\}$ of $H R^{\rm ML}_n$.
The following micro-lemma will be necessary to control the error in the gain.
\begin{lemma}[Multilevel covariance approximation error]  
\label{lem:mlcae}
Let $\tilde{R}^{\rm ML}_n$ be given by~\eqref{eq:covzee1}. Then 
there is a $0<c<\infty$ such that  
\begin{equation}
\rrNorm{H(\tilde{R}^{\rm ML}_n - {R}^{\rm ML}_n)} \leq c \hhNorm{C^{\rm ML}_n - \mfc_n}.
\label{eq:mlcov_bound}
\end{equation}
\end{lemma}

\begin{proof}
Notice that, by equivalence of $2$ and HS norms over $\bbR^m$, there exists a $0<\tilde c<\infty$ such that 
\begin{equation}
\rrNorm{H(\tilde{R}^{\rm ML}_n - {R}^{\rm ML}_n)} \leq   \tilde c {\rm max}_{\{j; \lambda_j<0\}} \{|\lambda_j|\}.
\end{equation}
Denote the associated eigenvector by $q_{\rm max}$ (normalized to 
{$\|q_{\rm max}\|=1$).
Notice that for any $A=A^\transpose\in \bbR^{m\times m}$ we can define 
\[
\|A\| := {\rm sup}_q \frac{|q^\transpose A q|}{\|q\|^2} = {\rm max}_i |\lambda_i|,
\]}
where $\lambda_i$ are the eigenvalues of $A$.

Since $\mfc_n \geq 0$, one has that 
\begin{eqnarray*}
\abs{q_{\rm max}^\transpose H ({C}^{\rm ML}_n - \mfc_n) H^* q_{\rm max}}  & = &
q_{\rm max}^\transpose H \mfc_n H^* q_{\rm max}
- q_{\rm max}^\transpose H C^{\rm ML}_n H^* q_{\rm max} \\
&\geq& \frac{1}{{\tilde c}} \rrNorm{H(\tilde{R}^{\rm ML}_n - {R}^{\rm ML}_n)}.
\end{eqnarray*}
The fact that for self-adjoint $Q:\cH \rightarrow \cH$ one has 
$\rrNorm{HQH^*} \leq \rhNorm{H}^2 \hhNorm{Q}$ concludes the proof.
\end{proof}

The next step is to bound the Kalman gain error in terms of the
covariance error.
\begin{lemma}[Kalman gain error]
\label{lem:gce}
There is a constant $\tilde c_n<\infty$, depending on $\rhNorm{H}, \gamma_{\rm min}$, 
and $\hhNorm{\mfk_n H}$ such that 
\begin{equation}
\hrNorm{K^{\rm ML}_n - \mfk_n}
\leq \tilde c_n \hhNorm{{C}^{\rm ML}_{n}-\mfc_{n}}. 
\label{eq:gaincov}
\end{equation}
\end{lemma}

\begin{proof}
It is shown in Lemma 3.4 of \cite{ourmlenkf} that 
\begin{eqnarray}
\mfk_n - \kml{n} &=&   \mfk_n H \big ( \tilde{R}^{\rm ML}_{n} -R_n 
 \big)
\big( H \tilde{R}^{\rm ML}_{n} + \Gamma \big )^{-1} \\
&+& \big( (\mfc_{n} - {C}^{\rm ML}_{n}) H^*\big) 
\big( H \tilde{R}^{\rm ML}_{n} 
+ \Gamma\big )^{-1}. 
\label{eq:kay}
\end{eqnarray}
Note that $x^\transpose (\Gamma+B) x \geq x^\transpose \Gamma x \geq \gamma_{\rm min}$ 
for all $x \in \bbR^m$ whenever $B=B^\transpose \geq 0$, and this implies that 
$\rrNorm{(H \tilde{R}^{\rm ML}_{n} H^* + \Gamma)^{-1}} \leq 1/{\gamma_{\rm min}}$
where $\gamma_{\rm min} >0$ is the smallest eigenvalue of $\Gamma$.
It follows by \eqref{eq:mlcov_bound} that 
\begin{equation}\label{eq:kayer1}
\hrNorm{\mfk_n - \kml{n}} \leq \frac{1+ 2\hhNorm{\mfk_n H}}{\gamma_{\rm min}} 
\rhNorm{H}\hhNorm{ \mfc_{n} - {C}^{\rm ML}_{n}}.
\end{equation}
\end{proof}

\begin{theorem}
\label{thm:covspliteps}
Suppose Assumptions~\ref{ass:psilip} and~\ref{ass:mlrates} hold and
for any $\varepsilon>0$, let $L$ and $\{M_\ell\}_{\ell=0}^L$ be
defined as in Theorem~\ref{thm:main}. Then the following inequality
holds {for any $p\ge 2$ and $n \in \bbN$},
\begin{equation}
\|C^{\rm ML}_n - \mfc_n\|_{L^p(\Omega; \cH \otimes \cH)} \lesssim \varepsilon 
+ \|C^{\rm ML}_n - \mfc^{\rm ML}_n\|_{L^p(\Omega; \cH \otimes \cH)}.
\label{eq:covspliteps}
\end{equation}
\end{theorem}

\begin{proof}
  Let $\mfc^L_n$ denote the predicting covariance of the final $L^{th}$
  level limiting system at time $n$, in the sense that the forward map
  above is replaced by $\Psi_L$, but the gain comes from the continuum
  mean-field limiting system.  Furthermore, let $\mfc^{\rm ML}_n$
  denote the covariance associated to the multilevel ensemble
  $\{(\mfv^{\ell}_{n,i})_{i=1}^{M_\ell}\}_{\ell=1}^L$. Minkowski's inequality is used to split
\begin{equation}
\|C^{\rm ML}_n - \mfc_n\|_{p} \leq \|\mfc^L_n - \mfc_n\|_{p} + 
\|\mfc^{\rm ML}_n - \mfc^L_n\|_{p} + \|C^{\rm ML}_n - \mfc^{\rm ML}_n \|_{p},
\label{eq:covsplit}
\end{equation}
and each term will be dealt with in turn, in the following three lemmas.  
The proof of the 
theorem is concluded after Lemmas~\ref{lem:disccov} and~\ref{lem:iidcover} 
which bound the first two terms, respectively.  
\end{proof}

\begin{lemma}
\label{lem:disccov}
Suppose Assumptions~\ref{ass:psilip} and~\ref{ass:mlrates} hold
 and for any 
$\varepsilon>0$, let $L$ be defined as in Theorem~\ref{thm:main}.
Then the following inequalities hold for any $n\in \bbN$ and $p\ge 2$,
\begin{equation}\label{eq:disccov2}
\|\mfc^L_n- \mfc_n\|_{\cH \otimes \cH}
\lesssim \varepsilon,
\end{equation}
\begin{equation}\label{eq:disccov1}
\max\left(  \| \mfv_{n}^L -  \mfv_{n}\|_{L^p(\Omega; \cH)}, \| \hmfv_{n}^L -  \hmfv_{n}\|_{L^p(\Omega; \cH)}  \right) 
\lesssim \varepsilon,
\end{equation}
and 
\begin{equation}\label{eq:disccov3}
\max\left(  \| \mfv_{n}^\ell -  \mfv_{n}^{\ell-1}\|_{L^p(\Omega; \cH)}, \| \hmfv_{n}^\ell -  \hmfv_{n}^{\ell-1}\|_{L^p(\Omega; \cH)}  \right) 
\lesssim h_\ell^{\beta/2}, \quad \forall \ell \in \bbN. 
\end{equation}

\end{lemma}
\begin{proof}
  The initial data for the respective mean-field methods is given by
  $\hmfv_0$ and $\hmfv^L_0 \eqcolon \cP_L  \hmfv_0$. 
Assumption~\ref{ass:mlrates}(ii) implies that 
\[
\|\hmfv_0 -  \hmfv^L_0 \|_{p} \lesssim h_L^{\beta/2} \lesssim \varepsilon.
\]
By \hh{Assumptions~\ref{ass:psilip}(i) and~\ref{ass:mlrates}(i),}
\[
\|\mfv_n -  \mfv^L_n \|_{p} \lesssim  \|\hmfv_{n-1} -  \hmfv^L_{n-1} \|_{p} + h_L^{\beta/2},
\]
and by Proposition \ref{prop:asimp}(iii),
\[
\begin{split}
 \|\hmfv_n -  \hmfv^L_n \|_{p}  
&\leq  \hhNorm{I-\mfk_{n}H} \|\mfv_{n}^L -  \mfv_{n}\|_p 
+ \|(I-\cP_L)\mfk_{n}(H\mfv_{n}^L + y_n) \|_p \\
& \leq c \left( \|\mfv_{n}^L -  \mfv_{n}\|_p + \|(I-\cP_L)\bar C_{n}\|_{\cH \times \cH} \right) \\
& \lesssim \|\mfv_{n}^L -  \mfv_{n}\|_p + \varepsilon,
\end{split}
\]
where $\bar S_n \coloneq (H\bar{C}_nH^* + \Gamma)$.
Inequality~\eqref{eq:disccov1} consequently holds by induction.  Furthermore, 
\[
\begin{split}
 &\|\mfc^L_n- \mfc_n\|_{\cH \otimes \cH}\\
 &= \hhNorm{\E{ (\mfv^L_n- \E{\mfv^L_n}) \otimes (\mfv^L_n- \E{\mfv^L_n}) 
 - (\mfv_n- \E{\mfv_n}) \otimes (\mfv_n- \E{\mfv_n})} }\\
&\leq  \norm{ (\mfv^L_n- \E{\mfv^L_n}) \otimes (\mfv^L_n- \E{\mfv^L_n}) 
 - (\mfv_n- \E{\mfv_n}) \otimes (\mfv_n- \E{\mfv_n}) }_{1} \\
& \leq   (\| \mfv^L_n- \E{\mfv^L_n} \|_2 +  \| (\mfv_n- \E{\mfv_n}) \|_2 ) \| \mfv^L_n - \mfv_n\|_2\\
& \lesssim \varepsilon. 
\end{split}
\]
An analogous argument may be used to bound the second term of inequality~\eqref{eq:disccov2}.

To prove inequality~\eqref{eq:disccov3}, note first that due to the matching initial data,
the inequality holds trivially for the update at $n=0$.
By Assumption~\ref{ass:psilip}(i), Proposition~\ref{prop:asimp}(i),
and Minkowski's inequality,
\[
\begin{split}
\norm{\mfv_{n}^\ell -  \mfv_{n}^{\ell-1}}_p
& \leq \norm{ \Psi^\ell( \hmfv_{n-1}^\ell) - \Psi^{\ell-1}( \hmfv_{n-1}^\ell)}_p
+ \norm{ \Psi^{\ell-1}(\hmfv_{n-1}^\ell) - \Psi^{\ell-1}(
  \hmfv_{n-1}^{\ell-1})}_p \\
& \lesssim \norm{ \hmfv_{n-1}^\ell -  \hmfv_{n-1}^{\ell-1}}_p + h_\ell^{\beta/2},
\end{split}
\]
and by \hh{Proposition \ref{prop:asimp}(iii),} 
\[
\begin{split}
\norm{\hmfv_{n}^\ell - \hmfv_{n}^{\ell-1}}_p &\leq 
\hhNorm{I-\cP_\ell\mfk_{n}H} \norm{\mfv_n^\ell -  \mfv_n^{\ell-1} }_p + 
\|(\cP_\ell-\cP_{\ell-1})\mfk_{n}H\mfv_{n}^{\ell-1} \|_p\\
&\lesssim \norm{\mfv_n^\ell -  \mfv_n^{\ell-1} }_p +  
(\|(I-\cP_\ell)\bar C_{n}\|_{\cH\otimes \cH}+
\|(I-\cP_{\ell-1})\bar C_{n}\|_{\cH\otimes \cH}) 
\|\mfv_{n}^{\ell-1} \|_{p}\\
&\lesssim \norm{\mfv_n^\ell -  \mfv_n^{\ell-1} }_p +  h_\ell^{-\beta/2}.
\end{split}
\]
Inequality~\eqref{eq:disccov3} holds by induction.

\end{proof}

Next we derive a bound for $\|\mfc^{\rm ML}_n - \mfc^L_n\|_p$, where we will 
make use of the following representation of the finte resolution mean-field covariance
\[
\mfc^L_n =\sum_{\ell=0}^L \cov[\mfv^\ell_n] - \cov[\mfv^{\ell-1}_n].
\]
We also recall that $\mfc^{\rm ML}$ denotes the mean-field MLEnKF sample covariance
defined by 
\begin{equation}\label{eq:covML-def}
\mfc^{\rm ML}_n = \sum_{\ell=0}^L \cov_{M_\ell}[\mfv^\ell_n] - \cov_{M_\ell}[\mfv^{\ell-1}_n].
\end{equation}

\begin{lemma}[Multilevel \iid sample covariance error]  
\label{lem:iidcover}
Suppose Assumptions~\ref{ass:psilip} and~\ref{ass:mlrates} 
hold and for any $\varepsilon>0$,  
let $L $ and $\{M_\ell\}_{\ell=0}^L$ be defined as in Theorem~\ref{thm:main}.
Then the following inequality holds for any $n \in \bbN$ and $p \ge 2$, 
\begin{equation}
\|\mfc^{\rm ML}_n - \mfc^L_n\|_{L^p(\Omega; \cH \otimes \cH)}   \lesssim \varepsilon.
\label{eq:iidcover}
\end{equation}
\end{lemma}

\begin{proof}
Recall that \eqref{eq:covML-def} is unbiased, 
$\E{\mfc^{\rm ML}_n } = \mfc^L_n$, so
\begin{equation}
\begin{split}
 \|\mfc^{\rm ML}_n  &- \mfc^L_n \|_{p} =
 \|\mfc^{\rm ML}_n  - \E{\mfc^{\rm ML}_n}\|_{p}.
\end{split}
\end{equation}
For a random field $Y: \Omega \to \cH$ we introduce the shorthand $\underline{Y} := Y - \E{Y}$. 
By equation~\eqref{eq:covML-def},
\begin{equation}
 \begin{split}
 \|\mfc^{\rm ML}_n  - \E{\mfc^{\rm ML}_n}\|_{p}
&= \bigg\|\sum_{\ell = 0}^L \underline{\bigg(\cov_{M_\ell}[\mfv^\ell_n] - \cov_{M_\ell}[\mfv^{\ell-1}_n] \bigg)}\bigg\|_{p} \\
&\leq \sum_{\ell = 0}^L \big\|\underline{\cov_{M_\ell}[\mfv^\ell_n] - \cov_{M_\ell}[\mfv^{\ell-1}_n] } \big\|_{p}\\
&\leq \sum_{\ell = 0}^L \left( \big\|\underline{\cov_{M_\ell}[\mfv^\ell_n, \Delta_\ell \mfv_n]} \big\|_{p}
+ \big\|\underline{\cov_{ M_\ell}[\Delta_\ell \mfv_n, \mfv^{\ell-1}_n]} \big\|_{p}
\right),
 \end{split}
\end{equation}
where we recall that $\Delta_\ell \mfv_n = \mfv_n^\ell - \mfv_n^{\ell-1}$. 
We have
\begin{equation*}
\underline{\cov_{M_\ell}[\mfv^\ell_n, \Delta_\ell \mfv_n]}
= \cov_{M_\ell}[\mfv^\ell_n, \Delta_\ell \mfv_n] -  \cov[\mfv^\ell_n, \Delta_\ell \mfv_n],
\end{equation*}
and similarly for the other term. By Lemmas~\ref{lem:disccov} and~\ref{lem:covM-Lr-error},
\begin{equation*}
 \begin{split}
\|\mfc^{\rm ML}_n  - \E{\mfc^{\rm ML}_n}\|_{p}
&\leq 2\sum_{\ell = 0}^L \frac{c}{\sqrt{M_\ell}}
( \|\mfv_n^\ell\|_{2p} + \|\mfv_n^{\ell-1}\|_{2p}) 
\| \Delta_\ell \mfv_n 
\|_{2p} \\
&\lesssim 
\sum_{\ell = 0}^L \frac{1}{\sqrt{M_\ell}} \|\Delta_\ell \mfv_n \|_{2p}
 \lesssim  \sum_{\ell = 0}^L M_\ell^{-1/2}h_\ell^{\beta/2}
\lesssim \varepsilon.
 \end{split}
\end{equation*}
\end{proof}

The previous two lemmas complete the proof of
Theorem~\ref{thm:covspliteps}.  We now turn to bounding the latter
term of the right-hand side of inequality~\eqref{eq:covspliteps}, the
difference between multilevel ensemble covariances.

\begin{lemma}
\label{lem:mlcov}

Suppose Assumptions~\ref{ass:psilip} and~\ref{ass:mlrates} hold and
for any $\varepsilon>0$,  let $L$ and $\{M_\ell\}_{\ell=0}^L$ be defined as in Theorem~\ref{thm:main}.
Then, {for any $p\ge 2$ and $n \in \bbN$}, 
\begin{equation}
\begin{split}
 \|C^{\rm ML}_n - \bar{C}^{\rm ML}_n\|_{L^p(\Omega; \cH \otimes \cH)} \leq &
4 \sum_{l=0}^L \|v_{n}^{\ell} - \mfv_{n}^\ell\|_{L^{2p}(\Omega, \cH)} (\|v_{n}^{\ell}\|_{L^{2p}(\Omega, \cH)}  + \|\mfv_{n}^{\ell}\|_{L^{2p}(\Omega,\cH)}).
\label{eq:mlcov}
\end{split}
\end{equation}
\end{lemma}

\begin{proof}
From the definitions of the sample covariance~\eqref{eq:sampleCov} and multilevel sample covariance~\eqref{eq:mlSampleMoments},
one obtains the bounds
\[
\begin{split}
\|C^{\rm ML}_n - \bar{C}^{\rm ML}_n\|_{p} & \leq 
 \sum_{\ell=0}^L \Big( \|\cov_{M_\ell}[v_n^\ell] - \cov_{M_\ell}[\mfv_n^\ell]\|_{p}  \\
&+ \|\cov_{M_\ell}[v_n^{\ell-1}] - \cov_{M_\ell}[\mfv_n^{\ell-1}]\|_{p}\Big), 
\end{split}
\]
and
\[
\begin{split}
\norm{\cov_{M_\ell}[v_n^\ell] - \cov_{M_\ell}[\mfv_n^\ell]}_{p}
&\leq \norm{E_{M_\ell}[v_n^\ell \otimes v_n^\ell] - E_{M_\ell}[\mfv_n^\ell \otimes \mfv_n^\ell]}_{p} \\
 & + \norm{E_{M_\ell}[v_n^\ell] \otimes E_{M_\ell}[ v_n^\ell] - E_{M_\ell}[\mfv_n^\ell ] \otimes E_{M_\ell}[\mfv_n^\ell]}_{p}\\
& \eqcolon I_{1} + I_{2}.
\end{split}
\]
The bilinearity of the sample covariance yields that
\begin{equation}\label{eq:I1}
I_1 \leq \norm{E_{M_\ell}[(v_n^\ell - \mfv_n^\ell) \otimes v_n^\ell] }_{p}
 + \norm{E_{M_\ell}[\mfv_n^\ell \otimes (v_n^\ell - \mfv_n^\ell) ]}_{p}
\end{equation}
and 
\begin{equation*}\label{eq:I2}
I_2 \leq \norm{E_{M_\ell}[(v_n^\ell - \mfv_n^\ell)]  \otimes E_{M_\ell}[v_n^\ell] }_{p} 
 + \norm{E_{M_\ell}[\mfv_n^\ell\hh{]} \otimes E_{M_\ell}[(v_n^\ell - \mfv_n^\ell) ]}_{p}.
\end{equation*}
For bounding $I_1$ we use Jensen's and H\"older's inequalities: 
\[
\begin{split}
\norm{E_{M_\ell}[(v_n^\ell - \mfv_n^\ell) \otimes v_n^\ell] }_{p}^p
& = \E{ \hhNorm{E_{M_\ell}[(v_n^\ell - \mfv_n^\ell)  \otimes v_n^\ell]}^p}\\
& \leq \E{E_{M_\ell}\Big[\hNorm{v_n^\ell - \mfv_n^\ell}^p  \hNorm{v_n^\ell}^p\Big]} \\
& = \E{ \hNorm{v_n^\ell - \mfv_n^\ell}^p  \hNorm{v_n^\ell}^p} \\
& \leq \norm{v_n^\ell - \mfv_n^\ell}_{2p}^p \norm{v_n^\ell}_{2p}^p.
\end{split}
\]
The second summand of inequality~\eqref{eq:I1} is bounded similarly, and we obtain
\[
I_1 \leq \norm{v_n^\ell - \mfv_n^\ell}_{2p} 
\parenthesis{\norm{v_n^\ell}_{2p} + \norm{\mfv_n^\ell}_{2p}}.
\]
The $I_2$ term can also be bounded with similar steps 
as in the preceding argument so that also
\[
I_2 \leq \norm{v_n^\ell - \mfv_n^\ell}_{2p}  
\parenthesis{\norm{v_n^\ell}_{2p }+ \norm{\mfv_n^\ell}_{2p}}.
\]
The proof is finished by summing the contributions of $I_1$ and $I_2$ over all levels.
\end{proof}

It has just been shown that the second term of \eqref{eq:covspliteps} is ``close in the predicting ensembles".
Therefore, the error level of the first term will carry over between observation times by induction.  
This is made rigorous by the next lemma.

\begin{lemma}[Distance between ensembles.] 
\label{lem:ensdist}
Suppose Assumptions~\ref{ass:psilip} and~\ref{ass:mlrates} hold and
for any $\varepsilon>0$,  let $L$ and $\{M_\ell\}_{\ell=0}^L$ be defined as in Theorem~\ref{thm:main}.
Then the following inequality holds for any $n \in \bbN$ and $p \ge 2$,
\begin{equation}
\sum_{\ell=0}^L  \|\hv_{n}^{\ell}- \hmfv_{n}^{\ell}\|_{L^p(\Omega; \cH)} \lesssim |\log(\varepsilon)|^n\varepsilon.
\label{eq:ensdist}
\end{equation}
\end{lemma}
\begin{proof}

We use an induction argument. Notice first of all that by definition,
\[
\|v_{0}^{\ell}- \mfv_{0}^{\ell}\|_{p} = 0
\]
Assume that for $p\geq 2$,
\begin{equation}
\sum_{\ell=0}^L \norm{\hv_{n-1}^{\ell}- \hmfv_{n-1}^{\ell}}_{p} \lesssim |\log(\varepsilon)|^{n-1}\varepsilon.
\label{eq:ind1ens}
\end{equation}
By Assumption~\ref{ass:psilip}(i),
the following inequality \hh{holds for the prediction ensemble}:
\begin{equation}
\sum_{\ell=0}^L \|v_n^{\ell} - \mfv_n^\ell\|_{p} 
\leq  \sum_{\ell=0}^L c_{\Psi} \|\hv_{n-1}^{\ell} - \hmfv_{n-1}^\ell\|_{p} \lesssim |\log(\varepsilon)|^{n-1}\varepsilon. 
\label{eq:indlip1ens}
\end{equation}
Furthermore, by 
Lemma \ref{lem:gce}, 
\begin{equation}\label{eq:appearanceOfCn}
\begin{split}
\hNorm{\hv_n^{\ell} - \hmfv_n^\ell}  &\leq   
 \hNorm{v_n^{\ell} - \mfv_n^\ell} \\
& +  \tilde c_n  \hhNorm{C^{\rm ML}_n - C_n}
\Big(\hNorm{v_n^{\ell} - \mfv_n^\ell} + \hNorm{y_n^{\ell} - \mfv_n^\ell}\Big),
\end{split}
\end{equation}
for all $\ell=0, \ldots, L$.
H{\"o}lder's inequality then implies
\begin{equation*}
\begin{split}
&  \|\hv_n^{\ell} - \hmfv_n^\ell\|_{p}  \leq   
 \|v_n^{\ell} - \mfv_n^\ell\|_{p} \\
&  \qquad +  \tilde c_n  \|{C}^{\rm ML}_n - \mfc_n\|_{2p}  
\Big ( \|v_n^{\ell} - \mfv_n^\ell \|_{2p} 
+ \|y_n^{\ell} - \mfv_n^\ell\|_{2p} \Big ).  
\end{split}
\end{equation*}
Plugging the moment bounds~\eqref{eq:indlip1ens} 
into the right-hand side of the inequality~\eqref{eq:mlcov}
yields that
$\|C^{\rm ML}_n - \bar{C}^{\rm ML}_n\|_{2p} \lesssim |\log(\varepsilon)|^{n-1}\varepsilon$,
which in combination with Theorem~\ref{thm:covspliteps}
leads to
$\|C^{\rm ML}_n - {C}_n\|_{2p} \lesssim |\log(\varepsilon)|^{n-1}\varepsilon$.
Therefore, summing the above and using \eqref{eq:indlip1ens} again for $p, 2p$ 
\begin{align*}
\nonumber
 \sum_{\ell=0}^L \|\hv_n^{\ell} - \hmfv_n^\ell\|_{p} 
& \lesssim 
\sum_{\ell=0}^L  \Big\{ \|v_n^{\ell} - \mfv_n^\ell\|_{p}  
 + \varepsilon  \left( \|v_n^{\ell} -    \mfv_n^\ell\|_{2p} +  
 \|y_n^{\ell} - \mfv_n^\ell\|_{2p} \right) \Big\} \\
& \lesssim |\log(\varepsilon)|^{n-1} \varepsilon \Big(1+ \sum_{\ell=0}^L\|y_n^{\ell} -
  \mfv_n^\ell\|_{2p}\Big)\\
& \lesssim |\log(\varepsilon)|^{n}\varepsilon.
\end{align*}

\end{proof}

Induction is complete on the distance between the multilevel ensemble
and its i.i.d.~shadow in $L^p(\Omega; \cH)$, and we are finally ready to 
prove the main result. 

\begin{proof}[Proof of Theorem~\ref{thm:main}] 

  
By Minkowski's inequality, 
\begin{align}
\nonumber
\|\hat \mu^{\rm ML}_n (\varphi) - \hat \mfmu_n (\varphi)\|_{p} & \leq 
\|\hat \mu^{\rm ML}_n (\varphi) - \hat \mfmu^{\rm ML}_n (\varphi)\|_{p}  
 + \|\hat\mfmu^{\rm ML}_n (\varphi) - \hat \mfmu^L_n(\varphi)\|_{p} \\
&+ \|\hat \mfmu^L_n(\varphi) - \hat \mfmu_n (\varphi)\|_{p},
\label{eq:lpertri}
\end{align}
where $\hat \mfmu^{\rm ML}_n$ denotes the empirical measure associated to 
the mean-field multilevel ensemble, and $\hat \mfmu^L_n$ denotes the probability measure 
associated to $\hmfv^L$. Before treating each term separately, we notice that the 
two first summands of the right-hand side of 
the inequality relates to the statistical error, whereas the last relates to the bias.

{By the global Lipschitz continutity of the observable $\varphi$,
Minkowski's inequality, and Lemma~\ref{lem:ensdist} the first term satisfies the following bound 
\begin{equation} \label{eq:finensembles}
\begin{split}
\norm{\hat \mu^{\rm ML}_n (\varphi) - \hat \mfmu^{\rm ML}_n (\varphi)}_{p} 
&= \left\| \sum_{\ell=0}^L E_{M_\ell}\big[\varphi(\hv_n^{\ell}) - \varphi(\hv_n^{\ell-1}) - 
(\varphi(\hmfv_n^\ell) - \varphi(\hmfv_n^{\ell-1}))\big] \right\|_{p}\\
& \leq \sum_{\ell=0}^L \parenthesis{\norm{ \varphi(\hv_n^{\ell}) - \varphi(\hmfv_n^\ell) }_{p}
+ \norm{ \varphi(\hv_n^{\ell-1}) - \varphi(\hmfv_n^{\ell-1}) }_{p}} \\
& \leq C_\varphi \sum_{\ell=0}^L \parenthesis{\norm{ \hv_n^{\ell} - \hmfv_n^\ell }_{p}
+ \norm{ \hv_n^{\ell-1} - \hmfv_n^{\ell-1} }_{p}} \\
& \lesssim |\log(\varepsilon)|^n\varepsilon.
\end{split}
\end{equation}

For the second summand~{of~\eqref{eq:lpertri}}, notice that we can
write $\hat \mfmu^L_n = \sum_{\ell=0}^L \hat \mfmu^\ell_n - \hat
\mfmu^{\ell-1}_n$, where $\hat \mfmu^\ell_n$ is the measure associated
to the level $\ell$ limiting process $\hmfv^\ell$ and $\mfmu^{-1}_n \coloneq 0$.
Then, by virtue of
Lemmas ~\ref{lem:disccov} and~\ref{lem:mz} and the global Lipschitz continuity
of $\varphi$,
\begin{equation}\label{eq:finvar}
 \begin{split}
\norm{\hat{\bar{\mu}}^{\rm ML}_n (\varphi) - \hat \mfmu^L_n(\varphi)}_{p}  & \leq
\sum_{\ell=0}^L \norm{E_{M_\ell}\Big[\varphi(\hmfv_n^\ell) - \varphi(\hmfv_n^{\ell-1}) - 
\E{\varphi(\hmfv_n^\ell) - \varphi(\hmfv_n^{\ell-1})} \Big]}_{p}  \\
 & \leq c \sum_{\ell=0}^L M_\ell^{-1/2}\norm{ \varphi(\hmfv_n^\ell) - \varphi(\hmfv_n^{\ell-1}) }_{p} \\
  & \leq \tilde c \sum_{\ell=0}^L M_\ell^{-1/2}\| \hmfv_n^\ell - \hmfv_n^{\ell-1}\|_{p} \\
& \lesssim \sum_{\ell=0}^L M_\ell^{-1/2}h_\ell^{\beta/2} \lesssim \varepsilon.
 \end{split}
 \end{equation}
}
Finally, the bias term in~\eqref{eq:lpertri} satisfies 
\begin{equation}\label{eq:finbias}
\begin{split}
\|\hat \mfmu^L_n(\varphi) - \hat \mfmu_n (\varphi)\|_{p} 
 =  |\hat \mfmu^L_n(\varphi) - \hat \mfmu_n (\varphi)| 
 = \abs{\E{\varphi(\hmfv^L_n) - \varphi(\hmfv_n)} }
\lesssim \varepsilon,
\end{split}
\end{equation}
where the last step follows from the Lipschitz property and Lemma~\ref{lem:disccov}.

Inequalities~\eqref{eq:finensembles}, \eqref{eq:finvar}, and \eqref{eq:finbias}
together with inequality~\eqref{eq:lpertri} complete the proof.
\end{proof}

Theorem~\ref{thm:main} shows the cost-to-$\varepsilon$ performance of
MLEnKF. 
The geometrically growing logarithmic penalty in the error 
\eqref{eq:lperror} is disconcerting.
The same penalty appears in the work \cite{ourmlenkf}, 
yet the numerical results there indicate a time-uniform rate of convergence,
and this may be an artifact of the rough bounds. 
We believe ergodicity of the MFEnKF process would 
allow us to obtain linear growth or even a uniform bound.
There has been much recent work in this direction.  
The interested reader
is referred to the works \cite{del2016stability, del2016stability1, majda2016rigorous, majda2016robustness, tong2016nonlinear}.

We conclude this section with a comparable result on the
cost-to-$\varepsilon$ perfomance of EnKF, showing that
MLEnKF generally outperforms EnKF.

\begin{theorem}[EnKF accuracy vs. cost]\label{thm:mainEnKF}
  {Consider a globally Lipschitz continuous observable function
  $\varphi:\cH \to \bbR$, and suppose Assumptions~\ref{ass:psilip}
  and~\ref{ass:mlrates} hold. For a given $\varepsilon >0$, let $L$
  and $M$ be defined under the respective constraints
  $L =  \lceil 2\log_\kappa(\varepsilon^{-1})/\beta\rceil$ and $ M \eqsim \varepsilon^{-2}$.
  Then, for any $n \in \bbN$ and $p \ge 2$,}
\begin{equation}
\|\hat{\mu}^{\rm MC}_n (\varphi) - \hat{\bar{\mu}}_n (\varphi) \|_p \lesssim \varepsilon,
\label{eq:lperrorEnKF}
\end{equation}
where $\hat{\mu}^{\rm MC}_n$ denotes the EnKF empirical measure,
cf.~equation~\eqref{eq:emp}, with particle evolution given by the EnKF
predict and update formulae at resolution level $L$ (i.e., with the numerical
integrator $\Psi^L$ in the prediction and projection operator $\cP_L$ in the update).

The computational cost of the EnKF estimator over the
time sequence is bounded by
\begin{equation}\label{eq:mlenkfCostsEnKF}
\cost{\mathrm{EnKF}} \lesssim \varepsilon^{-2(1+d\gamma/\beta)}.
\end{equation}
\end{theorem}

\begin{proof}[Sketch of proof]
By Minkowski's inequality,
\[
\begin{split}
\|\hat{\bar{\mu}}_n(\varphi) -\hat{\mu}^{\rm MC}_n (\varphi) \|_p & \leq 
\norm{\hat{\bar{\mu}}_n(\varphi) -  \hat{\bar{\mu}}_n^L (\varphi)}_p
+\norm{\hat{\bar{\mu}}^L(\varphi) - \hat{\bar{\mu}}_n^{\rm MC} (\varphi) }_p \\
& +  \norm{\hat{\bar{\mu}}_n^{\rm MC} (\varphi) -\hat{\mu}^{\rm MC}_n (\varphi)}_p
 \eqcolon I + II + III,
\end{split}
\]
where $\hat{\bar{\mu}}_n^{\rm MC}$ denotes the empricial measure associated
to the EnKF ensemble $\{\hmfv^L_{n,i} \}_{i=1}^M$ and
$\hat{\bar{\mu}}^{L}_n$ denotes the emprical measure associated to
$\hmfv^L_n$. 
It follows by inequality~\eqref{eq:finbias} that $I \lesssim \varepsilon$.

{For the second term, 
{note that \eqref{eq:localLipschitz} guarantees the existence of
a positive scalar $C_\varphi$ such that
$|\varphi(x)| \leq C_\varphi (1+\hNorm{x})$}.
Since $\hat{\bar{v}}_n^L \in L^p(\Omega; \cH)$ for any $n \in \bbN$ and $p\ge 2$,
it follows by Lemma~\ref{lem:mz} (on the Hilbert space $\cR_1$) that
\begin{equation*}
\begin{split}
II \leq \norm{E_{M}[ \varphi(\hmfv^L_n)] - \E{\varphi(\hmfv^L_n)} }_p  \leq M^{-1/2} C_\varphi\norm{\hmfv^L_n }_p 
\lesssim \varepsilon.
\end{split}
\end{equation*}

For the last term, let us first assume that for any $p \ge 2$ and 
 $n \in \bbN$, 
\begin{equation}\label{eq:termIIIAssumption}
\norm{\hv_{n,1}^L -  \hmfv_{n,1}^L }_p \lesssim \varepsilon,
\end{equation}
for the {single particle dynamics $\hv_{n,1}^L$ and $\hmfv_{n,1}^L$ respectively
 associated to the EnKF ensemble $\{\hv_{n,i}^L \}_{i=1}^M$ and 
the mean-field EnKF ensemble $\{\hmfv_{n,i}^L\}_{i=1}^M$.
Then the global Lipschitz continuity of $\varphi$, the fact
that $\hv_{n,1}^L, \hmfv_{n,1}^L \in L^p(\Omega; \cH)$}
for any $n\in \bbN$ and  $p\ge 2$,
and Minkowksi's inequality yield that
\[
\begin{split}
III = \norm{ E_{M}[ \varphi(\hv_n^L) - \varphi(\hmfv_n^L)] }_p
\leq C_\varphi \norm{\hv_n^L - \hmfv_n^L }_p 
\lesssim \varepsilon.
\end{split}
\]}
\noindent All that remains is to verify~\eqref{eq:termIIIAssumption}, but we
omit this verification as it can be done by similar steps as in
the proof of inequality~\eqref{eq:ensdist}.

\end{proof}

\section{A concrete example} 
\label{sec:example}

Consider the stochastic heat equation, given  
abstractly as follows:
\begin{equation}
du = - A u dt + B dW, \quad u(0) \sim N(0, C_0),
\label{eq:she}
\end{equation}
where $A$ is the abstract representation of $(-\Delta)$ acting on the space
$\cH := \{u \in L^2(D); \int_{D} u(x) dx = 0\}$, $B=A^{-b}$ 
for some $b \geq 0$ and $C_0=A^{-a}$ 
for some $a \geq 0$.  Let $D=[-\pi,\pi]^d$.  
The standard Sobolev spaces are defined as follows.
{\begin{definition}\label{def:sobolev}
  $\cH^s$ is defined as the space
  $\{ u \in \cH ; \langle u, A^s u \rangle_\cH { < \infty}\}$, 
  with the associated norm $\|u\|_{\cH^s} = \langle u, A^s u \rangle_\cH$.
\end{definition}
}
Consider the Fourier 
basis $\{\phi_k\}_{k=-\infty}^\infty$ such that $\phi_k(x)=e^{-ik\cdot x}$, $i=\sqrt{-1}$, and for $u\in\cH$
one has 
the expansion $u=\sum_{k=-\infty}^\infty u_k \phi_k(x)$ subject to reality constraint $u_{-k} = u_{k}^*$, and with 
$u_0=0$. One has spectral expansions 
$A= \sum_{k=-\infty}^\infty |k|^2 \phi_k \otimes \phi_k$,
$B= \sum_{k=-\infty}^\infty b_k \phi_k \otimes \phi_k$, and
$C= \sum_{k=-\infty}^\infty c_k \phi_k \otimes \phi_k$.  
The solution for $u_k$, $k\geq1$, is given analytically as 
\begin{equation}
u_k(t) = e^{-k^2t} u_k(0) + \xi_k(t), \quad 
\xi_k(t) \sim N\left[0,\frac{b^2_k}{2 k^2}(1-e^{-2k^2t})\right] \perp u_k(0).
\label{eq:she_exactmode}
\end{equation}
For observation increment $\tau$, the observations are taken as 
\begin{equation}
y_n = H u(\tau n) + \eta_n, \quad \eta_n \sim N(0,\Gamma) ~~ {\rm i.i.d.} \perp u(0) , 
\xi_k(\tau n) ~\forall~ k.
\label{eq:she_obs}
\end{equation} 
{The observation operator may be taken as
$H u = [H_1(u), \dots, H_m(u)]^\transpose$, where
$H_i(u) = \int u(x)\psi_i(x) dx$ for some $\psi_i\in \cH$.
Notice the model is non-trivial as
correlations will arise from the update unless $\psi_i=\phi_k$ for some $k$.}  


Note that for this simple Gaussian model one simply requires that 
$u \in L^2(\Omega;\cH)$, 
since all other moments are controlled by the variance.  
Indeed if $\bbE \|u\|^2_\cH <\infty$, 
then $u\in L^p(\Omega;\cH)$ for all $p$ 
and $u \in \cH$ almost surely.

Following from \eqref{eq:she_exactmode} define 
\begin{equation}
\Psi(u) := \sum_{k=0}^\infty\left ( e^{-k^2} u_k + \xi_k \right ) \phi_k, \quad \xi_k \sim N\left[0,\frac{b^2_k}{2 k^2}(1-e^{-2k^2})\right],
\label{eq:psiex}
\end{equation}
where $u_k=\langle \phi_k,u\rangle$. 
Notice that the regularity of $\Psi(u)$ does not depend on  $u$ at all (assuming it is not exponentially rough).
Indeed by the assumed form of $B$, one has $b_k = \cO(k^{-2b})$. 
Notice 
$$\bbE \|\Psi(u)\|^2_{\cH^{s}} = \sum_{k=1}^\infty k^{2s}
\left( e^{-2k^2} \bbE u^2_k + \frac{1}{2 k^{2(2b+1)}}(1-e^{-2k^2})\right).$$ 
Therefore, 
$\Psi(u) \in \cH^s$
for any $s<2b+1-d/2$. 

Indeed, $\Psi(u)$ is Gaussian with a {\it smoothing} covariance $C$, 
such that for $u\in \cH^s$, one has $C u \in \cH^{s+2 b+1}$.  
Assuming that $H:\cH \rightarrow \bbR^m$ 
is defined by $\cH$ inner products, 
then $H^*: \bbR^m \rightarrow \cH^*=\cH$.  
Hence the Kalman gain $K : \bbR^m \rightarrow \cH^{2b+1} \subset \cH$, 
following from the form of \eqref{eq:mfupdate}.  

For a concrete example, let $d=1$ and $b=0$.  
Then $\cH:=L^2(D)$, and $u_n\in L^p(\Omega;\cH)$ for all {$n\in \bbN$ 
and $p\geq 2$.}
Assume $\cP_\ell$ is the projection onto $2^\ell$ Fourier modes.  
The $k^{th}$ mode is given by $u_{n,k} = N(u_{n-1,k}e^{-k^2}, \sigma_k^2)$,
where $\sigma_k^2 = \cO(k^{-2})$.
This in turn induces a rate of convergence of 
$$
\|(I-\cP_\ell) u_n \|_{L^2(\Omega;\cH)} = 
\cO\left( \left(\sum_{\{k>2^\ell\}} \sigma_k^2\right)^{1/2}\right) 
= \cO\left(2^{-\ell/2}\right),
$$ 
as $\ell \rightarrow \infty$. {Higher moments follow from Gaussianity, 
with a $p-$dependent constant.}
The other assumptions are easily verified as well.

\section{Conclusion}
\label{sec:conclusion}

An extension of the recent work \cite{ourmlenkf} to spatially extended models is presented here, 
using a hierarchical decomposition based on the spatial resolution parameter.
The proof follows closely that of \cite{ourmlenkf}, except with the important extension to 
infinite-dimensions.
It is shown that {an optimality rate similar to 
vanilla MLMC} 
can extend to the case of sequential inference 
using EnKF for spatial models as well.
One may therefore expect that value can be leveraged, 
for a fixed computational cost,
by spreading work across a multilevel ensemble associated to models of multiple spatial 
resolutions rather than restricting to an ensemble associated only to the finest resolution model and using one very small ensemble.
This has potential for broad impact across application areas in which there has been a recent explosion
of interest in EnKF, for example weather prediction and subsurface exploration.  

\appendix

\section{Marcinkiewicz--Zygmund inequalities for Hilbert spaces}
For closing the proof of Lemma~\ref{lem:iidcover} we make use a couple
of lemmas extending the Marcinkiewicz--Zygmund inequality to separable
Banach spaces.
\begin{lemma}\cite[Theorem 5.2]{kwiatkowski2014convergence}
\label{lem:mz}
Let $2 \leq p < \infty$ and $X_i \in L^p(\Omega;\cH)$ be \iid samples of $X \in L^p(\Omega;\cH)$. Then 
\begin{equation}\label{eq:LawLargeNum-r}
\|E_M[X] - \E{X}\|_{L^p(\Omega;\cH)} \leq \frac{c_p}{\sqrt{M}} \|X - \E{X}\|_{L^p(\Omega;\cH)}
\end{equation}
where $c_p$ only depends on $p$.
\end{lemma}
\begin{proof}
Let $r_1, r_2, \ldots$ denote a sequence of real-valued \iid random variables with $P(r_i= \pm 1) = 1/2$.
A Banach space $\cK$ is said to be of R-type $q$ if there exists a $c>0$ such that for  every $\bar n \in \bbN$
and for all (deterministic) $x_1, x_2, \ldots, x_{\bar n} \in \cK$,
\[
\E{ \Big\|\sum_{i=1}^{\bar n} r_i x_i \Big\|_{\cK}} \leq c \parenthesis{\sum_{i=1}^{\bar{n}} \|x_i\|_{\cK}^q}^{1/q}.
\]
It is clear that all Hilbert spaces (and for our interest $\cH$, in particular) are of R-type 2, since their norms are 
induced by an inner product. Following the proofs of~\cite[Proposition 2.1 and Corollary 2.1]{woyczynski1980},
we introduce the symmetrization $\widetilde X_i \coloneq (X_i - X_i')$  and derive that
\begin{multline*}
\E{ \hNorm{\sum_{i=1}^{\bar n} X_i - \E{X} }^p} \leq  \E{ \hNorm{\sum_{i=1}^n \widetilde X_i }^p} 
= \E{ \hNorm{\sum_{i=1}^{\bar n} r_i \widetilde X_i}^p}\\
\leq c \E{ \parenthesis{\sum_{i=1}^{\bar n} \hNorm{\widetilde X_i  }^2}^{p/2}} 
\leq c 2^p \, \E{ \parenthesis{\sum_{i=1}^{\bar n} \hNorm{X_i -\E{X}}^2}^{p/2}}.
\end{multline*}
And by another application of H\"older's inequality,
\[
\begin{split}
\E{ \hNorm{\sum_{i=1}^{M} \frac{X_i - \E{X}}{M} }^p} 
&\leq \hat{c} M^{-p} \E{\parenthesis{\sum_{i=1}^{M} \hNorm{X_i-\E{X}}^2}^{p/2}} \\
&\leq \hat{c} M^{-p/2} \E{\hNorm{X-\E{X}}^p}. 
\end{split}
\]
\end{proof}

\begin{lemma}\label{lem:covM-Lr-error}
Suppose $X,Y \in L^p(\Omega;\cH)$, $p \geq 2$. 
Then, for $1 \leq r,s \leq \infty$ satisfying $1/r + 1/s = 1$, it holds that
\begin{equation}\label{eq:covM-Lr-error}
\begin{split}
\|\cov_M[X,Y] - \cov[X,Y]\|_{L^p(\Omega;\cH\otimes \cH)}
\leq \frac{c}{\sqrt{M}}
\|X\|_{L^{pr}(\Omega;\cH)} 
\|Y\|_{L^{ps}(\Omega;\cH)}
\end{split}
\end{equation}
where the upper bound for the constant 
$c = \dfrac{M}{M-1}\bigg(2c_p + \dfrac{c_{pr}c_{ps}+1}{\sqrt{M}}\bigg)$ 
only depends on $r,s$ and $p$.
\end{lemma}

\begin{proof}

Since $\cov[X,Y] = \cov[X-\E{X}, Y-\E{Y}]$ and $\cov_M[X,Y] = \cov_M[X-\E{X},Y-\E{Y}]$, cf.~\eqref{eq:sampleCov}, 
we may without loss of generality assume that $\E{X} = \E{Y} = 0$.
Using Minkowski's inequality
\begin{equation}
\begin{split}
& \frac{M-1}{M}\|\cov_M[X,Y] - \cov[X,Y]\|_p \\
& \leq 
\| E_M[X \otimes Y] - \E{X \otimes Y}\|_p  + \| E_M[X] \otimes E_M[Y] \|_p
 + \frac{1}{M}\|\E{X \otimes Y}\|_{\cH\otimes \cH}.
\end{split}
\end{equation}
We estimate the three terms in the right-hand side separately. 
Estimate \eqref{eq:LawLargeNum-r} and H\"older's inequality yield
\[
\begin{split}
\| E_M[X \otimes Y] - \E{X \otimes Y}\|_p
&\leq \frac{c_p}{\sqrt{M}} \|X\otimes Y - \E{X\otimes Y}\|_p \\
&\leq \frac{2c_p}{\sqrt{M}} \|X\otimes Y \|_p
 \leq \frac{2c_p}{\sqrt{M}} \|X\|_{pr} \|Y\|_{ps}.
\end{split}
\]
Similarly, since $\E{X} = \E{Y} = 0$ by assumption, 
we obtain by \eqref{eq:LawLargeNum-r} and H\"older's inequality
\begin{equation}
\begin{split}
\| E_M[X] \otimes E_M[Y] \|_{p}
\leq  
\|E_M[X]\|_{pr} 
\|E_M[Y]\|_{ps}
\leq \frac{c_{pr}c_{ps}}{M} 
\|X\|_{pr} 
\|Y\|_{ps}.
\end{split}
\end{equation}
And, finally, for the last term
\begin{equation*}
\frac{1}{M}\|\E{X \otimes Y}\|_{\cH\otimes \cH}
\leq \frac{1}{M}\|X \otimes Y\|_{L^1(\Omega,\cH\otimes \cH)}
\leq \frac{1}{M} \|X\|_{L^{pr}(\Omega;\cH)} \|Y\|_{L^{ps}(\Omega;\cH)}.
\end{equation*}
\end{proof}

{\bf Acknowledgements } Research reported in this publication was supported by the King Abdullah University of Science and Technology (KAUST).  
\hh{HH was additionally supported by Norges Forskningsr{\aa}d, research project 214495 LIQCRY}.  KJHL was additionally supported by an ORNL LDRD Strategic Hire grant.

\bibliography{mybib}

\def\cprime{$'$} \def\cprime{$'$} \def\cprime{$'$} \def\cprime{$'$}
  \def\cprime{$'$} \def\cprime{$'$} \def\cprime{$'$}
  \def\Rom#1{\uppercase\expandafter{\romannumeral #1}}\def\u#1{{\accent"15
  #1}}\def\Rom#1{\uppercase\expandafter{\romannumeral #1}}\def\u#1{{\accent"15
  #1}}\def\cprime{$'$} \def\cprime{$'$} \def\cprime{$'$} \def\cprime{$'$}
  \def\cprime{$'$} \def\cprime{$'$}
\begin{thebibliography}{10}

\bibitem{BC09}
A.~Bain and D.~Crisan.
\newblock {\em Fundamentals of Stochastic Filtering}.
\newblock Springer, 2009.

\bibitem{beskos2015multilevel}
Alexandros Beskos, Ajay Jasra, Kody Law, Raul Tempone, and Yan Zhou.
\newblock Multilevel sequential {Monte Carlo} samplers.
\newblock {\em To appear in Stochastic Processes and their Applications
  http://dx.doi.org/10.1016/j.spa.2016.08.004}.

\bibitem{blomker2013accuracy}
Dirk Bl{\"o}mker, Kody Law, Andrew~M Stuart, and Konstantinos~C Zygalakis.
\newblock Accuracy and stability of the continuous-time {3DVAR filter for the
  Navier--Stokes equation}.
\newblock {\em Nonlinearity}, 26(8):2193, 2013.

\bibitem{brett2013accuracy}
CEA Brett, Kei~Fong Lam, KJH Law, DS~McCormick, MR~Scott, and AM~Stuart.
\newblock Accuracy and stability of filters for dissipative {PDE}s.
\newblock {\em Physica D: Nonlinear Phenomena}, 245(1):34--45, 2013.

\bibitem{burgers1998analysis}
Gerrit Burgers, Peter Jan~van Leeuwen, and Geir Evensen.
\newblock Analysis scheme in the ensemble {Kalman} filter.
\newblock {\em Monthly weather review}, 126(6):1719--1724, 1998.

\bibitem{chung2001}
Kai~Lai Chung.
\newblock {\em A course in probability theory}.
\newblock Academic Press, Inc., San Diego, CA, third edition, 2001.

\bibitem{del2004feynman}
Pierre Del~Moral.
\newblock {\em Feynman-Kac Formulae: Genealogical and Interacting Particle
  Systems with Applications}.
\newblock Springer, 2004.

\bibitem{del2016multilevel}
Pierre Del~Moral, Ajay Jasra, Kody Law, and Yan Zhou.
\newblock Multilevel sequential {Monte C}arlo samplers for normalizing
  constants.
\newblock {\em arXiv preprint arXiv:1603.01136}, 2016.

\bibitem{del2016stability}
Pierre Del~Moral, Aline Kurtzmann, and Julian Tugaut.
\newblock On the stability and the uniform propagation of chaos of extended
  ensemble {Kalman-Bucy }filters.
\newblock {\em arXiv preprint arXiv:1606.08256}, 2016.

\bibitem{del2016stability1}
Pierre Del~Moral and Julian Tugaut.
\newblock On the stability and the uniform propagation of chaos properties of
  ensemble {Kalman-Bucy} filters.
\newblock {\em arXiv preprint arXiv:1605.09329}, 2016.

\bibitem{doucet2000sequential}
Arnaud Doucet, Simon Godsill, and Christophe Andrieu.
\newblock On sequential {{Monte Carlo}} sampling methods for {Bayesian}
  filtering.
\newblock {\em Statistics and computing}, 10(3):197--208, 2000.

\bibitem{evensen1994sequential}
Geir Evensen.
\newblock Sequential data assimilation with a nonlinear quasi-geostrophic model
  using {Monte Carlo} methods to forecast error statistics.
\newblock {\em Journal of Geophysical Research: Oceans (1978--2012)},
  99(C5):10143--10162, 1994.

\bibitem{evensen2003ensemble}
Geir Evensen.
\newblock The ensemble {Kalman} filter: Theoretical formulation and practical
  implementation.
\newblock {\em Ocean dynamics}, 53(4):343--367, 2003.

\bibitem{farhat2015data}
Aseel Farhat, Evelyn Lunasin, and Edriss~S Titi.
\newblock Data assimilation algorithm for {3D B}{\'e}nard convection in porous
  media employing only temperature measurements.
\newblock {\em arXiv preprint arXiv:1506.08678}, 2015.

\bibitem{Giles14}
M.~B. Giles and L.~Szpruch.
\newblock Antithetic multilevel {M}onte {C}arlo estimation for
  multi-dimensional {SDE}s without {L}\'evy area simulation.
\newblock {\em Ann. Appl. Probab.}, 24(4):1585--1620, 2014.

\bibitem{gregory2016multilevel}
Alastair Gregory, CJ~Cotter, and Sebastian Reich.
\newblock Multilevel ensemble transform particle filtering.
\newblock {\em SIAM Journal on Scientific Computing}, 38(3):A1317--A1338, 2016.

\bibitem{hayden2011discrete}
Kevin Hayden, Eric Olson, and Edriss~S Titi.
\newblock Discrete data assimilation in the lorenz and 2d navier--stokes
  equations.
\newblock {\em Physica D: Nonlinear Phenomena}, 240(18):1416--1425, 2011.

\bibitem{heinrich2001multilevel}
Stefan Heinrich.
\newblock {Multilevel Monte Carlo methods}.
\newblock In {\em Large-scale scientific computing}, pages 58--67. Springer,
  2001.

\bibitem{hoang2013complexity}
Viet~Ha Hoang, Christoph Schwab, and Andrew~M Stuart.
\newblock Complexity analysis of accelerated mcmc methods for {Bayesian}
  inversion.
\newblock {\em Inverse Problems}, 29(8):085010, 2013.

\bibitem{ourmlenkf}
H{\aa}kon Hoel, Kody Law, and Raul Tempone.
\newblock Multilevel ensemble {Kalman} filter.
\newblock {\em SIAM Journal of Numerical Analysis}, 54(3):1813--1839, 2016.

\bibitem{jasra2015multilevel}
Ajay Jasra, Kengo Kamatani, Kody~JH Law, and Yan Zhou.
\newblock Multilevel particle filter.
\newblock {\em arXiv preprint arXiv:1510.04977}, 2015.

\bibitem{jasra2016forward}
Ajay Jasra, Kody Law, and Yan Zhou.
\newblock Forward and inverse uncertainty quantification using multilevel
  {Monte C}arlo algorithms for an elliptic nonlocal equation.
\newblock {\em arXiv preprint arXiv:1603.06381}, 2016.

\bibitem{jaz70}
A.H. Jazwinski.
\newblock {\em Stochastic processes and filtering theory}, volume~63.
\newblock Academic Pr, 1970.

\bibitem{kalman1960new}
Rudolph~Emil Kalman et~al.
\newblock A new approach to linear filtering and prediction problems.
\newblock {\em Journal of basic Engineering}, 82(1):35--45, 1960.

\bibitem{kal03}
E.~Kalnay.
\newblock {\em Atmospheric Modeling, Data Assimilation and Predictability}.
\newblock Cambridge, 2003.

\bibitem{kelly2015concrete}
David Kelly, Andrew~J Majda, and Xin~T Tong.
\newblock Concrete ensemble {Kalman} filters with rigorous catastrophic filter
  divergence.
\newblock {\em Proceedings of the National Academy of Sciences},
  112(34):10589--10594, 2015.

\bibitem{kelly2014well}
DTB Kelly, KJH Law, and Andrew~M Stuart.
\newblock Well-posedness and accuracy of the ensemble {Kalman} filter in
  discrete and continuous time.
\newblock {\em Nonlinearity}, 27(10):2579, 2014.

\bibitem{ketelsen2013hierarchical}
C~Ketelsen, R~Scheichl, and AL~Teckentrup.
\newblock {A hierarchical multilevel Markov chain Monte Carlo algorithm with
  applications to uncertainty quantification in subsurface flow}.
\newblock {\em arXiv preprint arXiv:1303.7343}, 2013.

\bibitem{kwiatkowski2014convergence}
Evan Kwiatkowski and Jan Mandel.
\newblock Convergence of the square root ensemble {K}alman filter in the large
  ensemble limit.
\newblock {\em arXiv preprint arXiv:1404.4093}, 2014.

\bibitem{law2015data}
KJH Law, AM~Stuart, and KC~Zygalakis.
\newblock Data assimilation: A mathematical introduction.
\newblock {\em Springer Texts in Applied Mathematics}, 2015.

\bibitem{stuart2013analysis}
Kody~JH Law, Abhishek Shukla, and Andrew~M Stuart.
\newblock Analysis of the {3DVAR filter for the partially observed Lorenz'63
  model}.
\newblock {\em Discrete and Continuous Dynamical Systems}, 34(3):1061--1078,
  2013.

\bibitem{law2014deterministic}
Kody~JH Law, Hamidou Tembine, and Raul Tempone.
\newblock Deterministic mean-field ensemble {Kalman} filtering.
\newblock {\em SIAM Journal on Scientific Computing}, 38(3):A1251--A1279, 2016.

\bibitem{le2011large}
Fran{\c{c}}ois Le~Gland, Val{\'e}rie Monbet, Vu-Duc Tran, et~al.
\newblock Large sample asymptotics for the ensemble {Kalman} filter.
\newblock {\em The Oxford Handbook of Nonlinear Filtering}, pages 598--631,
  2011.

\bibitem{luenberger1968optimization}
David~G Luenberger.
\newblock {\em Optimization by vector space methods}.
\newblock John Wiley \& Sons, 1968.

\bibitem{majda2016rigorous}
Andrew~J Majda and Xin~T Tong.
\newblock Rigorous accuracy and robustness analysis for two-scale reduced
  random {Kalman} filters in high dimensions.
\newblock {\em arXiv preprint arXiv:1606.09087}, 2016.

\bibitem{majda2016robustness}
Andrew~J Majda and Xin~T Tong.
\newblock Robustness and accuracy of finite ensemble {Kalman} filters in large
  dimensions.
\newblock {\em arXiv preprint arXiv:1606.09321}, 2016.

\bibitem{mandel2011convergence}
Jan Mandel, Loren Cobb, and Jonathan~D Beezley.
\newblock On the convergence of the ensemble {Kalman} filter.
\newblock {\em Applications of Mathematics}, 56(6):533--541, 2011.

\bibitem{olson2003determining}
Eric Olson and Edriss~S Titi.
\newblock Determining modes for continuous data assimilation in 2d turbulence.
\newblock {\em Journal of statistical physics}, 113(5-6):799--840, 2003.

\bibitem{pajonk2012deterministic}
Oliver Pajonk, Bojana~V Rosi{\'c}, Alexander Litvinenko, and Hermann~G
  Matthies.
\newblock A deterministic filter for non-{Gaussian B}ayesian
  estimationÑapplications to dynamical system estimation with noisy
  measurements.
\newblock {\em Physica D: Nonlinear Phenomena}, 241(7):775--788, 2012.

\bibitem{tarn1976observers}
Tzyh-Jong Tarn and YONA Rasis.
\newblock Observers for nonlinear stochastic systems.
\newblock {\em IEEE Transactions on Automatic Control}, 21(4):441--448, 1976.

\bibitem{tong2015nonlinear}
Xin~T Tong, Andrew~J Majda, and David Kelly.
\newblock Nonlinear stability of the ensemble {Kalman} filter with adaptive
  covariance inflation.
\newblock {\em arXiv preprint arXiv:1507.08319}, 2015.

\bibitem{tong2016nonlinear}
Xin~T Tong, Andrew~J Majda, and David Kelly.
\newblock Nonlinear stability and ergodicity of ensemble based {K}alman
  filters.
\newblock {\em Nonlinearity}, 29(2):657, 2016.

\bibitem{woyczynski1980}
Wojbor~A. Woyczy{\'n}ski.
\newblock On {M}arcinkiewicz-{Z}ygmund laws of large numbers in {B}anach spaces
  and related rates of convergence.
\newblock {\em Probab. Math. Statist.}, 1(2):117--131 (1981), 1980.

\end{thebibliography}
\bibliographystyle{plain}

\end{document}